\documentclass[a4paper,12pt]{article}
\usepackage[left=30mm, top=30mm, right=30mm, bottom=30mm, nohead, nofoot]{geometry}
	
    \usepackage{float}
	\usepackage{amsmath,amsfonts,amssymb,amsthm,mathtools} 
	\usepackage{icomma} 
	
	\usepackage{wrapfig} 
	\usepackage{rotating}
    \usepackage{subfigure}
	\usepackage{hhline}
	\usepackage{lscape}
	\usepackage[usenames,dvipsnames,svgnames,table,rgb]{xcolor}
	
	\usepackage{enumitem}
	\setlist{nolistsep, leftmargin=5mm}
	
	\usepackage{array,tabularx,tabulary,booktabs} 
	\usepackage{longtable} 
	\usepackage{multirow} 
	
	\usepackage{multicol} 
	
	\usepackage{extsizes} 
	\usepackage{geometry} 
    \usepackage{hyperref}
	
	\usepackage{setspace} 
	\usepackage{lastpage} 
	\usepackage{soul} 
	\usepackage{bbding}
	\usepackage{hyperref}
	\usepackage[usenames,dvipsnames,svgnames,table,rgb]{xcolor}
	\hypersetup{ 
	colorlinks=true, 
	linkcolor=black, 
	citecolor=black, 
	filecolor=black, 
	urlcolor=ForestGreen 
	}
	
	\usepackage{environ}
	\makeatletter
	\newsavebox{\measure@tikzpicture}
	\NewEnviron{scaletikzpicturetowidth}[1]{%
	\def\tikz@width{#1}%
	\begin{lrbox}{\measure@tikzpicture}%
	\BODY
	\end{lrbox}%
	\pgfmathparse{#1/\wd\measure@tikzpicture}%
	\BODY
	}
	\makeatother
	
    \usepackage{verbatim}
	
	\usepackage{attachfile2}
	\attachfilesetup{appearance=true,
	color=0 0 0
	}
	
	\usepackage{forest} 
	\usepackage{vowel} 
	\usepackage{philex} 

	\usepackage{sectsty}
	\sectionfont{\normalsize}
	\subsectionfont{\normalsize}
	\usepackage{titlesec}
	\newlength{\bibitemsep}
	\newlength{\bibparskip}
	\let\oldthebibliography\thebibliography
	\renewcommand\thebibliography[1]{%
	\oldthebibliography{#1}%
	\setlength{\parskip}{\bibitemsep}%
	\setlength{\itemsep}{\bibparskip}%
	}
\usepackage{tikz}
\usepackage{pstricks}
\usetikzlibrary{arrows,positioning,shapes.geometric}
\usepackage{alltt}
\usepackage{amsthm}
\usepackage{amssymb}
\usepackage{amsmath}
\usepackage{tikz-cd}
\usepackage{hyperref}
\usepackage{xcolor}
\newtheorem*{Proposition}{Proposition}
\newtheorem{theorem}{Theorem}[section]
\newtheorem*{Theorem}{Theorem}
\newtheorem*{Corollary}{Corollary}

\newtheorem{corollary}[theorem]{Corollary}
\newtheorem{proposition}[theorem]{Proposition}
\newtheorem{fact}[theorem]{Fact}
\newtheorem{remark}[theorem]{Remark}
\newtheorem{definition}[theorem]{Definition}
\newtheorem*{notation}{Notation}

\newtheorem{example}[theorem]{Example}

\begin{document}

\title{On the action of the cactus group on the set of Gelfand-Tsetlin patterns for orthogonal Lie algebras}
\author{Igor Svyatnyy}

\definecolor{zzttqq}{rgb}{0.26666666666666666,0.26666666666666666,0.26666666666666666}
\definecolor{cqcqcq}{rgb}{0.7529411764705882,0.7529411764705882,0.7529411764705882}
\maketitle

\begin{abstract}
The purpose of this work is to define a natural action of the cactus group on the set of Gelfand-Tsetlin patterns for orthogonal Lie algebras. These Gelfand-Tsetlin patterns are meant to index the Gelfand-Tsetlin basis in the irreducible representations of the orthogonal Lie algebra $\mathfrak{o}_N$ with respect to the chain of nested orthogonal Lie algebras $\mathfrak{o}_N \supset \mathfrak{o}_{N-1} \supset \ldots \supset \mathfrak{o}_3$.
Using the Howe duality between $O_N$ and $\mathfrak{o}_{2n}$, we realize some representations of $\mathfrak{o}_N$ as multiplicity spaces inside the tensor power of the spinor representation $(\Lambda \mathbb{C}^{n})^{\otimes N}$. There is a natural choice of the basis inside the multiplicity space, which agrees with the decomposition of $(\Lambda \mathbb{C}^{n})^{\otimes N}$ into simple $\mathfrak{o}_{2n}$-modules. We call such basis \textit{principal}. The action of the cactus group $C_N$ by the crystal commutors on the crystal arising from $(\Lambda \mathbb{C}^{n})^{\otimes N}$ induces the action of $C_N$ on the set indexing the principal basis inside the multiplicity space. We call this set \textit{regular cell tables}. Regular cell tables are the analog of semi-standard Young tables. There is a natural bijection between a specific subset of semi-standard Young tables and regular cell tables. In this paper, we establish a natural bijection between the principal basis and the Gelfand-Tsetlin basis and, therefore, define an action of the cactus group on the set Gelfand-Tsetlin patterns.  
\end{abstract}

\tableofcontents

\section{Introduction}

\subsection{Tensor power of spinor representation.}

Let us define by $S$ the famous spinor representation of even orthogonal lie algebra $\mathfrak{o}_{2n}(\mathbb{C})$, where $n > 1$. As a vector space, $S$ is isomorphic to the exterior algebra of $n$-dimensional complex vector space. We will consider the $N$-th tensor power of this representation. It is well known that a finite-dimensional representation of a semisimple Lie algebra is completely reducible. Therefore, we can consider the decomposition of $S^{\otimes N}$ into the direct sum of irreducible subrepresentations.  $$S^{\otimes N} \cong \bigoplus_{\lambda \in \Delta^N } L_{\lambda}^{\oplus m_{\lambda}}.$$
Here $L_\lambda$ is the finite-dimensional irreducible representation of Lie algebra $\mathfrak{o}_{2n}(\mathbb{C})$ of the highest weight $\lambda$ , $\Delta^N $ is the set of highest weights of irreducible subrepresentations of $S^{\otimes N}$, $m_\lambda$ is the multiplicity of  $L_\lambda$ in $S^{\otimes N}$.

One may rewrite this decomposition in the following way.

$$S^{\otimes N} \cong  \bigoplus_{\lambda \in \Delta^N } U_{\lambda} \otimes L_{\lambda} \hspace{1mm},$$
where $U_\lambda = \text{Hom}_{\mathfrak{o}_{2n}}(L_{\lambda}, S^{\otimes N})$ is the multiplicity space of $L_\lambda$ in $S^{\otimes N}$.  We can also interpret $U_\lambda$ as the vector subspace of the singular vectors of weight $\lambda$ in $S^{\otimes N}$. 

We will explicitly describe the set of $\Delta^N $ and give a nice combinatorial interpretation of this set. In fact, in section \ref{sec_3} we are going to prove that the set $\Delta^N $ can be identified with the set of \textit{regular cell diagrams} of length $N$ and height $n$, denoted by $\mathfrak{D}(N,n)$. The definition of the regular cell diagram can be found in \ref{def_3.1}. 

The tensor product of spinor representation with any irreducible finite-dimensional representation of $\mathfrak{o}_{2n}(\mathbb{C})$ admits the following decomposition into the sum of irreducible representations without multiplicities:  

$$ L_\lambda \otimes S \cong \bigoplus_{\mu \in P[S] \hspace{1mm} : \hspace{1mm} \lambda + \mu \hspace{1mm}\in \hspace{1mm} P_+} L_{ \lambda + \mu} \hspace{1mm}, $$
where $P[S]$  is the set of weights  of spinor representation and $P_{+}$ is the set of dominant weights of $\mathfrak{o}_{2n}(\mathbb{C})$. 

Spinor representation $S$ is the sum of two irreducible representations $S \cong L_{\omega_+} \oplus L_{\omega_{-}}$, where one of them is spanned by all monomials of odd degree and another by all monomials of even degree. Therefore, from the above formula it immediately follows that $N$-th tensor power of spinor representation admits the following decomposition into the sum of irreducible subrepresentations without multiplicities:$$S^{\otimes N} \cong \bigoplus_{(\mu_1, \mu_2, \ldots, \mu_N) \in T^N} L_{\mu_1 + \mu_2 + \cdots + \mu_N} \hspace{1mm},$$
where $$T^N = \{(\mu_1, \mu_2, \ldots, \mu_N) \in P[S]^N \hspace{1mm}| \hspace{1mm} \mu_1 = \omega_{\pm}, \sum_{i=1}^{k} \mu_i \in P_+ \hspace{1mm}, \forall k \leqslant N   \}.$$

It means that there is a natural way to choose the decomposition of the isotypic components of $S^{\otimes N}$ into the sum of irreducible subrepresentations. We can choose the basis in the multiplicity space $U_\lambda$, which will agree with this decomposition. That is, if we identify $U_\lambda$ with the subspace of singular vectors of weight $\lambda$ in $S^{\otimes N}$, then we can define the \textit{principal basis} in $U_\lambda$, denoted by $\mathcal{PR}(U_\lambda)$, as the basis consisting of the highest weight vectors of the irreducible subrepresentations $L_{\mu_1 + \mu_2 + \ldots + \mu_N}$, such that $(\mu_1, \mu_2, \ldots, \mu_N) \in T^N $ and $ \mu_1 + \mu_2 + \ldots + \mu_N = \lambda$. 

\begin{notation}
$T_\lambda^N := \{(\mu_1, \mu_2, \ldots, \mu_N) \in T^N \hspace{1mm} | \hspace{1mm} \sum_{i = 1}^{N} 
\mu_i = \lambda \}$
\end{notation}

So, the principal basis of $U_\lambda$ is indexed by the set $T_\lambda^N$. It turns out that the set $T_\lambda^N$ has a purely combinatorial interpretation discussed in \ref{subs_cell_tables} as the set of \textit{regular cell tables} of the shape $D_\lambda^N$, denoted by $\mathfrak{Ctab}(D_\lambda^N)$, where $D_\lambda^N$ is a regular cell diagram corresponding to $\lambda \in \Delta^N $.  Therefore, $\mathcal{PR}(U_\lambda)$ is indexed by the set $\mathfrak{Ctab}(D_\lambda^N)$.

\subsection{Howe Duality. Gelfand-Tsetlin basis in multiplicity space.} 

As mentioned earlier, spinor representation as a vector space is isomorphic to the exterior algebra of $n$-dimensional vector space  $S \cong \Lambda E$, $\dim {E} = n$. Then there is an isomorphism of vector spaces $$S^{\otimes N} \cong \Lambda E^{\otimes N} \cong \Lambda (\mathbb{C}^{N} \otimes E)$$
Thus there is an action of the orthogonal group $O_N(\mathbb{C})$ on $S^{\otimes N}$:  it acts on $\mathbb{C}^{N} \otimes E$ by acting on the left factor and this action extends to action on the exterior algebra of this vector space. So, $S^{\otimes N}$ is a finite-dimensional (holomorphic) representation of $O_N(\mathbb{C})$. 

From the theorem $8$ in \cite{Howe} we deduce the following theorem.

\begin{Theorem}
    The action of  $O_N(\mathbb{C})$ on $S^{\otimes N}$ commutes with the action of $\mathfrak{o}_{2n}(\mathbb{C})$ and  $S^{\otimes N}$ under the joint action of $O_N(\mathbb{C})$ and $\mathfrak{o}_{2n}(\mathbb{C})$  decomposes into direct sum of irreducible representations without multiplicities: 
    $$S^{\otimes N} \cong  \bigoplus_{\lambda \in \Delta^N } U_{\lambda} \otimes L_{\lambda} \hspace{1mm},$$ where each $U_\lambda$ is an irreducible (holomorphic) representation of $O_N(\mathbb{C})$ and each $L_\lambda$ is the irreducible representation of the highest weight $\lambda$ of $\mathfrak{o}_{2n}(\mathbb{C})$. Moreover, $U_\lambda$ and $L_\lambda$ uniquely determine each other.  Equivalently speaking, multiplicity spaces $U_\lambda$ for different $\lambda \in \Delta^N $ are non-isomorphic as $O_N(\mathbb{C})$-modules. 
    
\end{Theorem}

The space $L_\lambda$ can be seen as the multiplicity space of irreducible subrepresentation $U_\lambda$ in $S^{\otimes N}$ (here we look at $S^{\otimes N }$ as a representation of $O_N(\mathbb{C}))$. As stated in the theorem,  $U_\lambda$ is an irreducible finite-dimensional (holomorphic) representation of $O_N(\mathbb{C})$. From the theorem $19.19$ in \cite{Fulton} such representations are indexed by Young diagrams in which the sum of the lengths of the first two columns does not exceed $N$.  We refer to them as \textit{short Young diagrams} of height $N$. We denote by $SYD(N,n)$ the set of short Young diagrams of height $N$ that have at most $n$ columns. In the section \ref{sec_5} we will prove the following theorem. 

\begin{Theorem}
 The map $$\mathcal{F}: \mathfrak{D}(N,n) \rightarrow SYD(N,n)$$ defined in \ref{th_5.1} is a bijection between the set of regular cell tables and the set short Young diagrams. Moreover, for each $\lambda \in \Delta^N $, the multiplicity space $U_\lambda$ as a representation of the group $O_N$ is isomorphic to the irreducible representation $R_{\nu}$, where $$\nu = \mathcal{F}(D_\lambda^N).$$
\end{Theorem}

Similarly to a semi-standard Young table, one can define a semi-standard short Young table. The precise definition is given in \ref{def_6.3}. 

We denote the set of semi-standard Young tables of the shape $\nu$ by $SSSYT(\nu)$. 

One can consider the Gelfand-Tsetlin basis inside the multiplicity space $U_\lambda$
with respect to the chain of nested orthogonal groups
$$O_N \supset O_{N-1} \supset \ldots \supset O_2 \supset O_1,$$
denoted by $GT(U_\lambda)$. 

It is easy to see that $GT(U_\lambda)$ is naturally indexed by the set $SSSYT(\mathcal{F}(D_\lambda^N))$.  We are going to prove the following proposition. 

\begin{Proposition}
    Gelfand-Tsetlin basis $GT(U_\lambda)$ up to rescaling is equal to the principal basis $\mathcal{PR}(U_\lambda)$ 
\end{Proposition}

Hence, the following corollary. 

\begin{Corollary}
There is a natural bijection defined in \ref{cor_6.4} between sets $\mathfrak{Ctab}(D_\lambda^N)$ and $SSSYT(\mathcal{F}(D_\lambda^N))$. It is given by the formula \ref{eq_6.10}.
\end{Corollary}

One can view $U_{\lambda}$ as a representation of Lie algebra $\mathfrak{o}_N(\mathbb{C})$.  Depending on $N$ and $\lambda$, $U_{\lambda}$ is either a simple $\mathfrak{o}_N(\mathbb{C})$-module or is the sum of two simple non-isomorphic $\mathfrak{o}_N(\mathbb{C})$ modules (their highest weights are uniquely determined by $N$ and $\lambda$). Either way, it makes sense to consider the Gelfand-Tsetlin basis of $U_{\lambda}$ with respect to the chain of nested orthogonal Lie algebras $$\mathfrak{o}_N \supset \mathfrak{o}_{N-1} \supset \ldots \supset \mathfrak{o}_{3},$$denoted by $\mathfrak{GT}(U_\lambda)$. 
This basis is indexed by the set of Gelfand-Tsetlin patterns with the first line determined by $\nu = \mathcal{F}(D_\lambda^N)$. We denote this set of Gelfand-Tsetlin patterns by $\mathfrak{GTP}(\nu)$.

Therefore, we obtain two bases in the multiplicity space $U_\lambda$: $\mathcal{PR}(U_\lambda) = GT(U_\lambda)$ and $\mathfrak{GT}(U_\lambda)$. These bases do not coincide, but in \ref{subs_bij} we construct a natural bijection between them; see \ref{prop_6.4}. Bijection between bases gives us a bijection between indexing sets.  

\begin{Proposition}
For each $\lambda \in \Delta^N $ there is a natural bijection between the set of Gelfand-Tsetlin patterns $\mathfrak{GTP}(\nu)$, where $\nu = \mathcal{F}(D_\lambda^N)$, and the set of regular cell tables  $\mathfrak{Ctab}(D_\lambda^N)$, given by the composition of the bijections constructed in \ref{cor_6.4}, \ref{cor_6.5}
\end{Proposition}

 \subsection{Cactus group and its action}

Let $C_N$ be a group with generators $s_{p, q}$, $1 \leqslant p < q \leqslant N$ and the defining relations 
$$s_{p,q}^{2} = e;$$
$$s_{p_1, q_1} s_{p_2, q_2} = s_{p_2, q_2} s_{p_1, q_1}, \hspace{1mm} \text{if} \hspace{2mm} q_1 < p_2;$$
$$s_{p_1, q_1} s_{p_2, q_2} s_{p_1, q_1} = s_{p_1 + q_1 - q_2, p_1 + q_1 - p_2}, \hspace{1mm} \text{if} \hspace{2mm} p_1 \leqslant p_2 < q_2 \leqslant q_1.$$
 
 There is an epimorphism $\pi : C_N \rightarrow S_N$ which takes generator $s_{p, q}$ to the involution reversing the segment $\{p, \ldots, q\} \subset \{1, \ldots, N\}$. Let $PC_N$ be the kernel of $\pi$. In the work of Herniques and Kamnitzer \cite{Herniques_Kamnitzer} the groups $C_N$ and $PC_N$ were named \textit{cactus group} and \textit{pure cactus group} respectively. 

Henriques and Kamnitzer observed that these groups naturally appear in \textit{coboundary categories}. Coboundary category is a monoidal category with a functorial involutive isomorphism $s_{X,Y}: X \otimes Y  \rightarrow  Y \otimes X$, called \textit{commutor}, satisfying certain natural relations. 

In this paper, we will work with the coboundary category of $\mathfrak{g}$-crystals. Roughly, a crystal is a directed graph where the edges are labeled by simple roots of $\mathfrak{g}$ and the vertices are labeled by weights of $\mathfrak{g}$. The tensor product of two crystals is a crystal whose underlying set is the product of the underlying sets of multipliers. 

For our purposes, we can assume that $\mathfrak{g}$ is a finite-dimensional semisimple Lie algebra. We are only interested in a special case $\mathfrak{g} = \mathfrak{o}_{2n}(\mathbb{C})$. Commutor for finite-dimensional reductive lie algebra $\mathfrak{g}$ was defined by Herniques and Kamnitzer in \cite{Herniques_Kamnitzer}. 

Category of $\mathfrak{g}$-crystals is semisimple. Simple objects in the category of $\mathfrak{g}$-crystals are indexed by the set of dominant weights of $\mathfrak{g}$ just like finite-dimensional irreducible representations of $\mathfrak{g}$. Informally, one can say that in a sense each crystal arises from some finite-dimensional representation of $\mathfrak{g}$.

The decomposition of the tensor product of $\mathfrak{g}$-crystals into the sum of simple crystals agrees with the decomposition of the tensor product of respective $\mathfrak{g}$-modules into the sum of simple modules. 

Now set $\mathfrak{g}$ to be equal to $\mathfrak{o}_{2n}(\mathbb{C})$. Let $\mathcal{B_{S}}$ be the crystal that arises from spinor representation. If we consider tensor degree of this crystal $\mathcal{{B}_{S}}^{\otimes N}$  we get the following decomposition:  $$\mathcal{{B}_{S}}^{\otimes N} =\bigoplus_{\lambda \in \Delta^N } \mathcal{B}_{\lambda}^{\oplus m_{\lambda}},$$ where $\mathcal{B}_{\lambda}$ is a simple crystal of the highest weight $\lambda$. 

Similarly to the case with representations, we find that the set $\mathfrak{Ctab}(D_\lambda^N)$ of cell tables of the shape $D_\lambda^N$ indexes the components $\mathcal{B}_{\lambda}$ in $\mathcal{{B}_{S}}^{\otimes N}$.     

The group $C_N$ acts on $\mathcal{{B}_{S}}^{\otimes N}$ by the suitable compositions of commutors. Commutors are isomorphisms of crystals, and hence they just permute the summands $\mathcal{B}_{\lambda}$ in $\mathcal{{B}_{S}}^{\otimes N}$ for all $\lambda \in \Delta^N $. 

\begin{Proposition}
For each $\lambda \in \Delta^N $ there is a natural action of the cactus group $C_N$ on the set of regular cell tables $\mathfrak{Ctab}(D_\lambda^N)$ given by the action of commutors on $\mathcal{{B}_{S}}^{\otimes N}$
\end{Proposition}
Therefore, we get the following corollary. 

\begin{Corollary}
For each $\lambda \in \Delta^N $ there is a natural action of the cactus group $C_N$ on the set of Gelfand-Tsetlin patterns $\mathfrak{GTP(\nu)}$ and on the set of semi-standard short Young tables $SSSYT(\nu)$, where $\nu = \mathcal{F}(D_\lambda^N)$. 
\end{Corollary}

Understanding the action of the cactus group $C_N$ on these sets will be the aim of my next study. 

\subsection{This paper is organized as follows}

\hspace{5mm} In the section \ref{sec_1} we give all the necessary definitions regarding spinor representation. 
In section \ref{sec_2} we prove the decomposition formula for the tensor power of spinor representation. The key results of this section are theorem \ref{th_2.1} and \ref{cor_2.1}.

In section \ref{sec_3} we introduce an important notion of a regular cell diagram and prove that the set of regular cell diagrams corresponds to the highest weights of the subrepresentations of a tensor power of spinor representation; see \ref{prop_3.2}. 

In section \ref{sec_4} we prove that multiplicity space $U_\lambda$ has the natural structure of an irreducible orthogonal group representation, using Howe's fundamental result \cite{Howe}.
In section \ref{sec_5} we find an isomorphism class of the multiplicity space $U_\lambda$ as $O_N$-module. The key result of the section is theorem \ref{th_5.1}. 

In section \ref{sec_6} we introduce the notion of a regular cell table and define 3 bases inside the multiplicity space $U_\lambda$. We also establish natural bijections between them and between their indexing sets. The main result of this section is corollaries \ref{cor_6.4} and \ref{cor_6.5}.

In section \ref{sec_7} we give all the necessary information on crystals. In section \ref{sec_8} we introduce the cactus group and define its action on the set of regular cell tables, short Young diagrams, and Gelfand-Tsetlin patterns, see \ref{cor_8.1} and \ref{cor_8.2}.

\subsection{Acknowledgments}
I am grateful to my scientific advisor Leonid Rybnikov for driving my attention to the subject and for the extremely useful discussions and references. I am happy to thank Grigory Olshansky for pointing out the relevance of Howe duality to my research and for providing all the necessary references. Finally, I would like to thank Sergey Davydov for the useful references. 

\section{Spinor representation and Quantum algebra}
\label{sec_1}
We will start by giving the standard but necessary definitions. 
Let $W$ be a $k$-dimensional vector space over the field of complex numbers. Denote by $W^{*}$ its dual. Let us fix some basis in $W$:  $\textbf{w} = (w_1, w_2, \ldots w_k)$. By $\textbf{w}^* = (w_1^*, w_2^*, \ldots, w_k^*)$ we denote the dual basis in $W^{*}$. For all $i, j \in \{ 1, 2, \ldots, k\}$ we have: $$\langle w_i, w_j^* \rangle = \delta_{ij}$$ 
Now take the exterior algebra $\Lambda W$. We will be interested in constructing some certain subalgebra in $End_{\mathbb{C}}(\Lambda W)$. 
For any vector $w_i$ from $\textbf{w}$ let us define the operator $M_{w_i} \in End_{\mathbb{C}}(\Lambda W)$ the following way: 
\begin{gather}
M_{w_i}: x \mapsto w_i \wedge x, \hspace{3mm} \forall x \in \Lambda W. 
\end{gather}
Similarly, for any vector $w_i^{*} \in \textbf{w}^{*}$ we define an operator $D_{w_i^*} \in$ $End_{\mathbb{C}}(\Lambda W)$:
\begin{gather}
D_{w_i^*}: w_i \wedge x \mapsto x; 
 \hspace{3mm} x \mapsto 0, \hspace{3mm} \forall x \in \Lambda W, \hspace{1mm} \text{such that} \hspace{1mm} w_i \wedge x \neq 0.
\end{gather}
\begin{definition} \label{def_1.1}
\textbf{Quantum subalgebra} of $\text{End}_{\mathbb{C}}(\Lambda W)$ is the subalgebra generated by the endomorphisms $M_{w_i}, D_{w_i^*}$ and $1_{\Lambda W}$. We denote it by $\text{End}_{\mathbb{C}}^{\circ}(\Lambda W)$. 
\end{definition}
\begin{notation}
    $\tilde{W} = W \oplus W^{*}$
\end{notation}
\begin{remark} \label{r_1.1}
There is a natural non-degenerate symmetric bilinear form $(\cdot, \cdot)$ on $\tilde{W}$ defined by the following equation: 
$$(x + x^{*}, y +  y^{*}) = \langle x, y^* \rangle + \langle x^*, y \rangle \hspace{1mm}, $$
$\forall x, y \in W, \hspace{1mm} \forall x^*, y^* \in W^*$. 
In basis $\tilde{\textbf{w}} := \textbf{w} \hspace{1mm} \cup \hspace{1mm} \textbf{w}^*$ this form has the following Gram matrix $G = \begin{bmatrix} 
0 & I_n \\
I_n & 0 
\end{bmatrix}.$
\end{remark}
\begin{remark}
\label{r_1.2}
Orthogonal Lie algebra $\mathfrak{o}(\tilde{W})$ as a matrix subalgebra of $\mathfrak{gl}(\tilde{W})$ written in the basis $\tilde{\textbf{w}}$ has the following form $\mathfrak{o}(\tilde{W}) = \{ x \in \mathfrak{gl}(\tilde{W}) \hspace{1mm}| \hspace{1mm} Gx + x^{t}G = 0 \}$. Equivalently, one may say that $\mathfrak{o}(\tilde{W})$ consists of the matrices $\begin{bmatrix} 
A & B \\
C & D 
\end{bmatrix}$ such that $A^{t} + D = 0, C + C^{t} = 0, B + B^{t} = 0$.
\end{remark}
\begin{proposition}
\label{prop_1.1}
Quantum subalgebra $\text{End}^{\circ}_{\mathbb{C}}(\Lambda W)$ is isomorphic to the Clifford algebra $CL(\tilde{W}, (\cdot, \cdot)_{\tilde{W}})$, where $(\cdot, \cdot)_{\tilde{W}}$ is the bilinear form on $\tilde{W}$ defined in \ref{r_1.1}. The isomorphism is defined by the map $\iota: CL(\tilde{W}) {\rightarrow} \text{End}^{\circ}_{\mathbb{C}}(\Lambda W)$  
\begin{gather}
    \iota(w_i) = M_{w_i}, \\
    \iota(w_i^*) = D_{w_i^*}, \\
    \iota(1_{CL}) = 1_{\Lambda W}
\end{gather}

\end{proposition}
Consider the vector subspace of $CL(\tilde{W})$ spanned by all the commutators of elements in $\tilde{W} \subset CL(\tilde{W})$. Denote it by $\mathfrak{g} = \text{span}_{\mathbb{C}}\{ab - ba \hspace{1mm}| \hspace{1mm} a, b \in \tilde{W} \}$. 

\begin{proposition}
\label{prop_1.2}  There is a lie algebra structure on $\mathfrak{g}$ defined by the commutator of Clifford Algebra $CL(\tilde{W})$. Notice that $\mathfrak{g}$ acts on $\tilde{W}$ by commutators. Therefore, we get a faithful representation of the Lie algebra $\mathfrak{g}$ denoted by $\rho: \mathfrak{g} \rightarrow \mathfrak{gl}(\tilde{W})$. Moreover, $\rho$ identifies $\mathfrak{g}$ with orthogonal lie algebra: $\rho(\mathfrak{g}) = \mathfrak{o}(\tilde{W})$
\end{proposition}
\begin{corollary}
\label{cor_1.1} On the external algebra $\Lambda W$ there is a structure of the $\mathfrak{o}(\tilde{W})$-module, given by the restriction of the mapping $\iota$ to $\mathfrak{g}\cong \mathfrak{o}(\tilde{W})$. 
\end{corollary}
\begin{definition}
 \label{def_1.2}   Representation $\Lambda W$ of the Lie algebra $\mathfrak{o}(\tilde{W}) \cong \mathfrak{o}_{2k}(\mathbb{C})$ is called \textbf{spinor} representation. 
\end{definition}
We are going to work with finite-dimensional representations of even orthogonal algebra. Therefore, it is essential to know the classification of finite-dimensional irreducible representations of $\mathfrak{o}(\tilde{W})$. We choose the Cartan subalgebra $\mathfrak{h}$ of $\mathfrak{o}(\tilde{W})$ to be the subalgebra of the diagonal (in basis $\tilde{\textbf{w}}$) matrices in $\mathfrak{o}(\tilde{W})$. 
$$\mathfrak{h} := \text{span}_{\mathbb{C}}\{h_i = E_{ii} - E_{i+k, i+k} \hspace{1mm} | \hspace{1mm} 1 \leqslant i \leqslant k\}.$$
We fix the basis in $\mathfrak{h}$: $$\textbf{h} = (h_i)_{1\leqslant i \leqslant k}$$
and the dual basis in $\mathfrak{h}^{*}$: $$\textbf{h}^{*} = (h_i^{*})_{1\leqslant i \leqslant k}$$
Root system in basis $\textbf{h}^{*}$ has the following form: $$R = \{ \pm h_i^* \pm h_j^*\ | \hspace{1mm} i \neq j \}$$
We choose polarization so that 
$$R_+ = \{h_i^* + h_j^* \hspace{1mm} |  \hspace{1mm} i \neq j \} \cup \{ h_i^* - h_j^* \hspace{1mm} | \hspace{1mm} i < j \}$$
Then $\mathfrak{o}(\tilde{W})$ admits triangular decomposition:  $$\mathfrak{o}(\tilde{W}) = \mathfrak{n}_{-} \oplus \mathfrak{h} \oplus \mathfrak{n}_{+},$$
where  $$\mathfrak{n}_{-} = \oplus_{\alpha \in R_{-}} \mathfrak{g}_{\alpha},$$ 
and $$\mathfrak{n}_{+} = \oplus_{\alpha \in R_{+}} \mathfrak{g}_{\alpha}.$$
Explicitly: 
$$\mathfrak{n}_{-} = \text{span}_{\mathbb{C}}\{ E_{i + k, j} - E_{j + k, i} \hspace{1mm} | \hspace{1mm} i < j \} \oplus \text{span}_{\mathbb{C}}\{  E_{ij} - E_{j+k, i+k} \hspace{1mm} | \hspace{1mm} i > j \},$$
 and
$$\mathfrak{n}_{+} = \text{span}_{\mathbb{C}}\{ E_{i, j + k} - E_{j, i + k} \hspace{1mm} | \hspace{1mm} i < j \} \oplus \text{span}_{\mathbb{C}}\{  E_{ij} - E_{j+k, i+k} \hspace{1mm} | \hspace{1mm} i < j \}.$$
\begin{proposition}
   \label{prop_1.3} Spinor representation $\Lambda W$ is isomorphic to $L_{\omega_{+}} \oplus L_{\omega_{-}}$  the direct sum of two irreducible representations with the highest weights $\omega_{\pm}$: $$ \Lambda W \cong L_{\omega_+} \oplus L_{\omega_-},$$
    where  $\omega_{\pm}= (\frac{1}{2}, \cdots, \frac{1}{2}, \pm \frac{1}{2})$.  
    Set of weights of spinor representation is equal to $$P[\Lambda W] = \left\{\left( \pm \frac{1}{2}, \pm\frac{1}{2}, \cdots, \pm\frac{1}{2} \right) \right\}$$  All weights are written in the basis $\textbf{h}^{*}$.
\end{proposition}

We adopt the following notation. 

\begin{notation}
    We denote by $s_{W}: \mathfrak{o}(\tilde{W}) \rightarrow \text{End}_{\mathbb{C}}(\Lambda W)$ the spinor representation with the base space $W$. 
\end{notation}
 
\begin{corollary}
\label{cor_1.2}  Using \ref{prop_1.1} and \ref{prop_1.2} one can describe $s_W$ explicitly: 
\begin{gather}
s_{W}(E_{ij} - E_{j+k, i+k})  = \frac{1}{2}(M_{w_i} \circ D_{w_j^{*}} - D_{w_j^{*}} \circ M_{w_i}), \\ 
s_{W}(E_{i + k, j} - E_{j + k, i}) = D_{w_i^{*}} \circ D_{w_j^{*}}, \\
s_{W}(E_{i, j + k} - E_{j, i + k}) = M_{w_i} \circ M_{w_j}.
\end{gather}
  \end{corollary}

In this paper, it is assumed that the reader knows basic facts about representations of finite-dimensional semisimple Lie algebras (representations and Lie algebras are assumed to be complex). All necessary information on this subject in the book of Kirillov \cite{Kirillov}.

\section{Tensor power of a spinor representation. Decomposition formula.} 
\label{sec_2}
Let $E$ be the $n$-dimensional vector space over the field of complex numbers. Consider $N$-th tensor power of a spinor representation with the base space $E$:  
$$s_E^{\otimes N}: \mathfrak{o}(\tilde{E}) \longrightarrow \text{End}_{\mathbb{C}}(\Lambda E^{\otimes N})$$

Every finite-dimensional representation of the algebra Lie $\mathfrak{o}(\tilde{E})$ is completely reducible. Hence, as an $\mathfrak{o}(\tilde{E})$-module $\Lambda E^{\otimes N}$ can be decomposed into the direct sum of simple $\mathfrak{o}(\tilde{E})$-modules: 
$$\Lambda E^{\otimes N} \underset{\mathfrak{o}(\tilde{E})}{\cong} \bigoplus_{\lambda \in P_+} L_\lambda^{\oplus m_\lambda},$$
where $P_+$ is the set of dominant weights of $\mathfrak{o}(\tilde{E})$ and $m_\lambda$ are non-negative integers.  

We rewrite this decomposition in the following way: 

\begin{equation}\label{eq_2.1} \Lambda E^{\otimes N} \underset{\mathfrak{o}(\tilde{E})}{\cong} \bigoplus_{\lambda \in P_+} U_{\lambda} \otimes L_\lambda,\end{equation}
where $U_\lambda$ is the vector space of dimension $m_\lambda$.  Algebra Lie $\mathfrak{o}(\tilde{E})$ acts on the right factor of each summand  $U_{\lambda} \otimes L_\lambda$. So, if we fix the basis $\textbf{u} = (u_1, u_2, \ldots u_{m_\lambda})$, we can establish the following non-canonical isomorphism of $\mathfrak{o}(\tilde{E})$ modules: $$U_{\lambda} \otimes L_\lambda \underset{\mathfrak{o}(\tilde{E})}{\cong} \bigoplus_{i = 1}^{m_{\lambda}} \mathbb{C}u_i \otimes L_\lambda \underset{\mathfrak{o}(\tilde{E})}{\cong} L_\lambda ^{\oplus m_\lambda}$$ We treat $U_\lambda$ as a \textbf{multiplicity space} of irreducible subrepresentation $L_\lambda$.  

A natural question to ask is for which $\lambda \in P_+$, $\dim U_\lambda = m_\lambda \neq 0$, i.e. $L_\lambda$ is isomorphic to some subrepresentation of $\Lambda E$. Also, one may ask if there is a formula for $m_\lambda$. To answer these questions, we need the following theorem. 

\begin{theorem} \label{th_2.1}
Let $L_\lambda$ be the irreducible representation of the algebra Lie $\mathfrak{o}_{2n}(\mathbb{C}) = \mathfrak{o}(\tilde{E})$ of the highest weight $\lambda \in P_+$. Let $P[\Lambda E]$ be the set of weights of the spinor representation $\Lambda E$  explicitly described in \ref{prop_1.3} Then the following decomposition formula is valid: 
\begin{equation} \label{eq_2.2}
     L_\lambda \otimes \Lambda E \underset{\mathfrak{o}(\tilde{E})}{\cong} \bigoplus_{\mu \in P[\Lambda E]: \hspace{1mm} \lambda + \mu \in P_+} L_{\lambda+\mu} 
\end{equation}
\end{theorem}
\begin{proof}
We assume that all weights are written in the basis $\textbf{h}^{*}$, defined in the previous section. The formula can be checked directly using the Weyl character formula for the irreducible representation $L_\lambda$. 
\begin{equation}
        ch(L_\lambda) = \frac{\sum_{w\in W} (-1)^{l(w)}e^{w(\lambda + \rho)}}{\prod_{\alpha \in R_+}(e^{\alpha/2} - e^{-\alpha/2})},
\end{equation}
where $W$ is the Weyl group, $l(w)$ is the length of the element $w \in W$ and (in type D) $$\rho = (n-1, n-2, \ldots 0).$$
As all weight subspaces of $\Lambda E$ are one-dimensional we can write the character of $\Lambda E$: 
$$ch(\Lambda E) = \sum_{\mu \in P[\Lambda E]} e^{\mu}
$$
Using the language of characters one can rewrite the statement of the theorem in the following way: 

$$ch(\Lambda E) \cdot ch(L_\lambda) = \sum_{\mu \in P[\Lambda E]: \hspace{1mm} \lambda + \mu \in P_+}ch(L_{\lambda + \mu})$$
So, we want to show that \begin{gather} \label{extra_1} \left(\sum_{\mu \in P[\Lambda E]} e^{\mu} \right)  \cdot \frac{\sum_{w\in W} (-1)^{l(w)}e^{w(\lambda + \rho)}}{\prod_{\alpha \in R_+}(e^{\alpha/2} - e^{-\alpha/2})}  = \sum_{\mu \in P[\Lambda E]: \hspace{1mm} \lambda + \mu \in P_+} \frac{\sum_{w\in W} (-1)^{l(w)}e^{w(\lambda + \mu + \rho)}}{\prod_{\alpha \in R_+}(e^{\alpha/2} - e^{-\alpha/2})} \end{gather}
It is clear that $$\sum_{\mu \in S} e^\mu = \sum_{\mu \in P[\Lambda E]} e^{w(\mu)},$$ for any $w \in W$. Using this equality, we can rewrite the left side of the equation \ref{extra_1} in the following way: 

$$\left(\sum_{\mu \in P[\Lambda E]} e^{\mu} \right)  \cdot \frac{\sum_{w\in W} (-1)^{l(w)}e^{w(\lambda + \rho)}}{\prod_{\alpha \in R_+}(e^{\alpha/2} - e^{-\alpha/2})} = \sum_{\mu \in P[\Lambda E]}\frac{\sum_{w\in W} (-1)^{l(w)}e^{w(\lambda + \mu + \rho)}}{\prod_{\alpha \in R_+}(e^{\alpha/2} - e^{-\alpha/2})}.$$
Hence, to finish the proof it is enough to show that \begin{gather} \label{extra_2}
 \sum_{\mu \in P[\Lambda E]: \lambda + \mu \notin P_+} \frac{\sum_{w\in W} (-1)^{l(w)}e^{w(\lambda + \mu + \rho)}}{\prod_{\alpha \in R_+}(e^{\alpha/2} - e^{-\alpha/2})} = 0.
 \end{gather}
 Notice that if $\lambda + \mu \notin P_+$, then $\exists j: s_{\alpha_j}(\lambda + \mu + \rho) = \lambda + \mu + \rho$, where $s_{\alpha_j}$ is one of the simple reflections that generate the Weyl group. Since $s_{\alpha_j}$ is involutive, all elements of the Weyl group can be divided into pairs $\{w, ws_{\alpha_j}\}$. Clearly, for any $w \in W$ $$|l(w) - l(ws_{\alpha_j})| = 1.$$ 
Therefore, if $\lambda + \mu \notin P_+$, then $$\sum_{w\in W} (-1)^{l(w)}e^{w(\lambda + \mu + \rho)} = \sum_{\{w, ws_{\alpha_j}\}} (-1)^{l(w)}e^{w(\lambda + \mu + \rho)} + (-1)^{l(ws_{\alpha_j})}e^{w(\lambda + \mu + \rho)} 
 = 0.$$
We get that each summand on the left side of the equation \ref{extra_2} is $0$ and we are done.  
\end{proof}

Using the formula \ref{eq_2.2}, we immediately get the following corollary. 

\begin{corollary}
\label{cor_2.1}
Tensor power of the spinor representation $\Lambda E^{\otimes N}$ of algebra Lie $\mathfrak{o}(\tilde{E})$ decomposes into the direct sum of simple $\mathfrak{o}(\tilde{E})$-modules :
\begin{gather} \label{eq_2.3}
    \Lambda E^{\otimes N}  \underset{\mathfrak{o}(\tilde{E})}{\cong} \bigoplus_{(\mu_1, \mu_2, \cdots, \mu_N) \in T^N} L_{\mu_1 + \mu_2 + \cdots + \mu_N}, \\
 \text{where} \hspace{1mm}   T^N = \{(\mu_1, \mu_2, \ldots, \mu_N) \in P[\Lambda E]^N \hspace{1mm}| \hspace{1mm} \mu_1 = \omega_{\pm}, \sum_{i=1}^{k} \mu_i \in P_+ \hspace{1mm}, \forall k \leqslant N   \}
\end{gather}
\end{corollary}
\begin{proof}
    Follows directly from the theorem \ref{th_2.1} and proposition \ref{prop_1.3}.
\end{proof}

\begin{notation}
    Denote the set of the highest weights of irreducible subrepresentations in $\Lambda E^{\otimes N}$ by $$\Delta^N  := \{\sum_{i=1}^N \mu_i \hspace{1mm} | \hspace{1mm} (\mu_1, \ldots, \mu_N) \in T^N \}.$$
\end{notation}

We can refine the decomposition formula \ref{eq_2.1}: 
\begin{gather}
    \label{eq_2.4} \Lambda E^{\otimes N} \underset{\mathfrak{o}(\tilde{E})}{\cong} \bigoplus_{\lambda \in \Delta^N } U_{\lambda} \otimes L_\lambda.
\end{gather}

\section{Cell diagrams}
\label{sec_3}
It happens that the set $\Delta^N $ has a nice combinatorial interpretation in terms of \textit{cell diagrams} (which is an analog of Young diagrams). 

\begin{definition}
\label{def_3.1}
\textbf{\textit{regular cell diagram}} of the length $N$ and of the height $n$ is a composition of cells and a vertical straight line, called \textbf{axis} of the cell diagram constructed by two sets of non-negative integers $\textbf{{l}} = (l_i)_{1 \leqslant i \leqslant n}$, $\textbf{{r}} = (r_i)_{1 \leqslant i \leqslant n}$ satisfying the following conditions:
\begin{itemize}
 \item $r_i + l_i = N, \hspace{1mm} \forall i \in \{1, \ldots n\}$
 \item $r_1 \geqslant r_2 \geqslant \ldots \geqslant r_n$
 \item $r_{n-1} \geqslant l_{n}$
\end{itemize}
The construction looks as follows. For every $i \in \{1, \ldots, n\}$ in the $i$-th row to the right of the axis draw $r_i$ cells and to the left of the axis draw $l_i$ cells. The enumeration of rows goes from top to bottom.
\end{definition}

\begin{notation}
The regular cell diagram, corresponding to sets $\textbf{{l}}, \textbf{{r}}$ is denoted by $D(\textbf{l}, \textbf{r})$.
\end{notation}
\begin{notation}
The set of regular cell diagrams of the length $N$ and height $n$ is denoted by  $\mathfrak{D}(N, n)$.
\end{notation}
For clarity, we provide an example of a regular cell diagram

\begin{example}
\label{ex_3.1}
Let $N =7, n = 4$. Take $\textbf{r} = (r_1, r_2, r_3, r_4) = (5, 4, 4, 3); \hspace{2mm} \textbf{l} = (l_1, l_2, l_3, l_4) = (2, 3, 3, 4)$. Then we get  regular cell diagram $D(\textbf{l}, \textbf{r})$: 

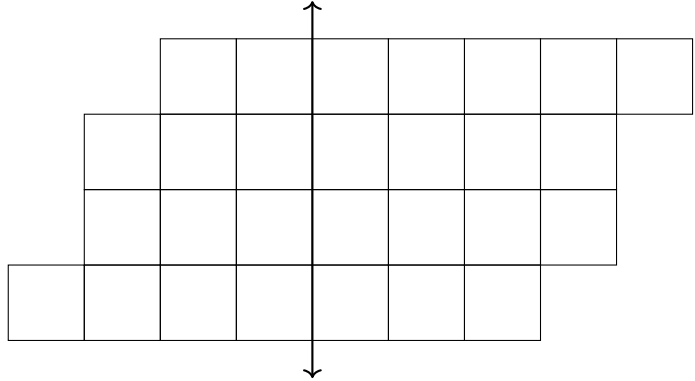
\centering 
\begin{tikzpicture}
\draw[thick,<->] (6,-0.5) -- (6, 4.5) node[anchor=north west] {};
\draw (2,0) rectangle (3,1);
\draw (3,0) rectangle (4,1);
\draw (4,0) rectangle (5,1);
\draw (5,0) rectangle (6,1);
\draw (6,0) rectangle (7,1);
\draw (7,0) rectangle (8,1);
\draw (8,0) rectangle (9,1);
\draw (3,1) rectangle (4,2);
\draw (4,1) rectangle (5,2);
\draw (5,1) rectangle (6,2);
\draw (6,1) rectangle (7,2);
\draw (7,1) rectangle (8,2);
\draw (8,1) rectangle (9,2);
\draw (9,1) rectangle (10,2);
\draw (3,2) rectangle (4,3);
\draw (4,2) rectangle (5,3);
\draw (5,2) rectangle (6,3);
\draw (6,2) rectangle (7,3);
\draw (7,2) rectangle (8,3);
\draw (8,2) rectangle (9,3);
\draw (9,2) rectangle (10,3);
\draw (5,3) rectangle (4,4);
\draw (6,3) rectangle (5,4);
\draw (7,3) rectangle (6,4);
\draw (8,3) rectangle (7,4);
\draw (9,3) rectangle (8,4);
\draw (10,3) rectangle (9,4);
\draw (11,3) rectangle (10,4);

\end{tikzpicture}

\end{example}

Our goal is to identify the set $\Delta^N $ with the set of regular cell diagrams of length $N$ and height $n$ denoted by $\mathfrak{D}(N, n)$. To do that, we need to make some observations about the set $\Delta^N $. 

From the very definition of $\Delta^N $ it is clear that $\Delta^N  \subset P_+$. For each $\lambda \in P_+$ and $N \in \mathbb{N}$ lets define two sets of numbers $$\textbf{{l}}(\lambda, N) = (l_i)_{1 \leqslant i \leqslant n}$$
and $$\textbf{{r}}(\lambda, N) = (r_i)_{1 \leqslant i \leqslant n}.$$ Each pair $(r_i, l_i), \forall i \in \{1, \ldots, n\}$ is uniquely determined as a solution of the following system of equations: \begin{gather} \label{eq_3.1}
\begin{cases} 
r_i + l_i = N \\
r_i - l_i = 2 \lambda_i
\end{cases}
\end{gather}
In other words: 
\begin{gather} \label{eq_3.2}
\begin{cases} 
r_i = \frac{N}{2} + \lambda_i \\
l_i = \frac{N}{2} - \lambda_i
\end{cases}
\end{gather}

\begin{proposition}
\label{prop_3.1}
For any $\lambda \in P_+$, $\lambda \in \Delta^N $ iff  $\textbf{{r}}(\lambda, N)$, $\textbf{{l}}(\lambda, N) \in (\mathbb{Z}_{+})^{\times n}.$
\end{proposition}
\begin{proof}
If $\lambda \in \Delta^N $ then there exists a tuple $(\mu_1, \mu_2 \ldots \mu_N) \in T^N$, such that $\sum_{j=1}^N \mu_j = \lambda$.  Then let  $\bar{r}_i$ be the amount of weights $\mu_j$ in this tuple with the $i$-th coordinate equal to $\frac{1}{2}$ and let $\bar{l}_i$ be the amount of weights $\mu_j$ in this tuple with the $i$-th coordinate equal to $-\frac{1}{2}$. It is easy to see that the pair $(\bar{r}_i, \bar{l}_i)$ is the solution of the system \ref{eq_3.1}. Hence, $r_i = \bar{r}_i \in \mathbb{Z}_{+}$ and $l_i = \bar{l}_i \in \mathbb{Z}_+$.

If $\textbf{{r}}(\lambda, N)$, $\textbf{{l}}(\lambda, N) \in (\mathbb{Z}_{+})^{\times n}$ then one can construct the tuple $(\mu_1, \mu_2, \ldots \mu_N)$ such that 

\begin{gather*}
    (\mu_j)_i = \begin{cases}
    \frac{1}{2}, \hspace{1mm} \text{if} \hspace{1mm} r_i \geqslant j \\
    -\frac{1}{2}, \hspace{1mm} \text{otherwise}
    \end{cases}
\end{gather*}
It is clear that $\sum_{j=1}^{N} \mu_j = \lambda$. We need to check that $(\mu_1, \mu_2, \ldots \mu_N) \in T^N$. Since $\lambda \in P_+$, it follows from \ref{eq_3.2} that $$r_1 \geqslant r_2 \geqslant \ldots \geqslant r_n$$ and that $$r_{n-1} \geqslant l_n.$$
We want to show that $\sum_{j = 1}^{k} \mu_j \in P_+$ for every $k \leqslant N$. It follows from the construction of the tuple that

$$\left(\sum_{j = 1}^{k} \mu_j\right)_i = \frac{\text{min}(k, r_i)}{2} - \frac{\text{max}(k - r_i, 0)}{2}.$$
Since 
$$r_1 \geqslant r_2 \geqslant \ldots \geqslant r_n,$$
we immediately get  $$\left(\sum_{j = 1}^{k} \mu_j\right)_1 \geqslant \left(\sum_{j = 1}^{k} \mu_j\right)_2 \geqslant \ldots \geqslant \left(\sum_{j = 1}^{k} \mu_j\right)_n. $$

It remains to check that $$ \left(\sum_{j = 1}^{k} \mu_j\right)_{n-1}+ \left(\sum_{j = 1}^{k} \mu_j\right)_{n} \geqslant 0.$$
We rewrite it as follows: 
\begin{gather} \label{eq_3.3} \frac{\text{min}(k, r_{n-1})}{2} - \frac{\text{max}(k - r_{n-1}, 0)}{2} + \frac{\text{min}(k, r_{n})}{2} - \frac{\text{max}(k - r_{n}, 0)}{2} \geqslant{0}.
\end{gather}
Notice that \begin{gather} \label{eq_3.4} \frac{\text{min}(k, r_{n-1})}{2} -  \frac{\text{max}(k - r_{n}, 0)}{2} \geqslant 0.\end{gather} Indeed, if $k < r_{n-1}$, then it is obviously true. On the contrary, if $k \geqslant r_{n-1}$, then we have to prove that $$\frac{r_{n-1}}{2} - \frac{\text{max}(k - r_{n}, 0)}{2} \geqslant 0,$$ 
which is true, since $$r_{n-1} \geqslant l_{n} = N - r_n  \geqslant k - r_n$$ Similarly, we get 

$$ \frac{\text{min}(k, r_{n})}{2} -  \frac{\text{max}(k - r_{n-1}, 0)}{2} \geqslant 0,$$

since  $r_{n-1} \geqslant l_{n} \Rightarrow  r_n \geqslant l_{n-1} =  N - r_{n-1} \geqslant k - r_{n-1}$. 

Hence, we get the inequality \ref{eq_3.3} and we are done. 
\end{proof}

Using proposition \ref{prop_3.1},  one can explicitly describe the set of $\Delta^N $

\begin{corollary}
  \label{cor_3.1}
Let $\lambda$ be an element of $P_+$. Then $\lambda \in \Delta^N $ iff  for every $i \in \{1, 2, \ldots n \}$:
 \begin{itemize}
 \item $-\frac{1}{2}N \leqslant \lambda_i  \leqslant  \frac{1}{2}N , $
 \item $(2\lambda_i + N) \mathrel{\vdots} 2 $.
 \end{itemize}

\end{corollary}

From the proof of the proposition \ref{prop_3.1} we know that for each $\lambda \in \Delta^N $ the sets $\textbf{{r}}(\lambda, N)$, $\textbf{{l}}(\lambda, N) \in (\mathbb{Z}_{+})^{\times n}$ and that $r_1 \geqslant r_2 \geqslant \ldots \geqslant r_n$ and $r_{n-1} \geqslant l_{n}$.  From the very definition of the sets $\textbf{{r}}(\lambda, N)$, $\textbf{{l}}(\lambda, N)$ we have  $r_i + l_i = N$ for every $i \in \{1, 2, \ldots n \}$. Hence, the following definition makes sense. 
 
\begin{definition}
\label{def_3.2} \textbf{Cell diagram} of the shape $\lambda \in \Delta^N $, denoted by $D^N_{\lambda}$, is a regular cell diagram of the length $N$ and height $n$ constructed by the sets $\textbf{{l}}(\lambda, N)$, $\textbf{{r}}(\lambda, N)$.  This definition can be expressed using the following formula: 
$$D_\lambda^N := D(\textbf{{l}}(\lambda, N),\textbf{{r}}(\lambda, N)).$$
\end{definition}

\begin{example}
\label{ex_3.2}
Let  $N = 7$ and $ \hspace{1mm} \lambda = (\frac{3}{2}, \frac{1}{2}, \frac{1}{2}, -\frac{1}{2})$ \newline
Then $D_{\lambda}^{N}$ is the regular cell diagram from the example \ref{ex_3.1}
\end{example}

Consider the map from $\Delta^N $ to the set $\mathfrak{D}(N, n)$, which sends $\lambda$ to $D_\lambda^{N}$. To identify these sets we need to prove that this map is a bijection.

\begin{proposition}
\label{prop_3.2}    
Map $\mathcal{K}_N : \Delta^N  \rightarrow \mathfrak{D}(N, n)$, where $$\mathcal{K}_N(\lambda) =  D_{\lambda}^N $$is a bijection. 
\end{proposition}
\begin{proof}
It is clear that $\mathcal{K}_N$ is injective. Indeed, if $D_\lambda^N = D_{\lambda^{'}}^N$ then $$\textbf{{l}}(\lambda^{'}, N) = \textbf{{l}}(\lambda, N)$$
and 
$$\textbf{{r}}(\lambda^{'}, N) = \textbf{{r}}(\lambda, N).$$
But then $$\lambda^{'}_i = \lambda_{i}, \hspace{1mm} \forall i \in \{1, 2, \ldots, n \} \Rightarrow \lambda^{'} = \lambda.$$
Lets prove the surjectivity of $\mathcal{K}_N$. Take a regular cell diagram $D(\textbf{l}, \textbf{r}) \in \mathfrak{D}(N, n)$. Define the coefficients of $\lambda = (\lambda_i)_{1 \leqslant i \leqslant n}$ to be equal to
$$\lambda_i =\frac{r_i - l_i}{2}.$$
It is clear that $\lambda \in P_+$. Since $D(\textbf{l}, \textbf{r}) \in \mathfrak{D}(N, n)$ we have 
$$r_i + l_i = N , \hspace{1mm} \forall i \in \{1, 2, \ldots, n\}.$$
Therefore, $$\textbf{r} = \textbf{{r}}(\lambda, N)$$ and 
    $$\textbf{l} = \textbf{{l}}(\lambda, N)$$
We conclude that $\lambda \in P_+$ and $\textbf{{r}}(\lambda, N), \textbf{{l}}(\lambda, N) \in (\mathbb{Z}_{+})^{\times n}$. From the proposition \ref{prop_3.1} it follows that $\lambda \in \Delta^{N} $. So, $$D(\textbf{l}, \textbf{r}) = D(\textbf{{l}}(\lambda, N),\textbf{{r}}(\lambda, N)) = D_\lambda^{N} = \mathcal{K}_N(\lambda)$$
 and we are done.
\end{proof}

\section{Howe duality.} 

\label{sec_4}

In this section, we will prove that multiplicity spaces $U_\lambda$ defined in paragraph $2$ (see \ref{eq_2.1} or \ref{eq_2.4}) have a natural structure of an irreducible representation of the orthogonal group $O_N(\mathbb{C})$. 

Let $V$ be a vector space of dimension $N$ over the field of complex numbers. There is the natural action of $O_N(\mathbb{C})$ on $V$. Consider the vector space $W = V^{\oplus n}$. As $O_N(\mathbb{C})$-module $W$ is isomorphic to $V \otimes E$, where $E$ is a $n$-dimensional vector space and the group acts on the left factor. Take a look at the space $$\tilde{W} = W \oplus W^{*} \cong V \otimes E \oplus V^{*} \otimes E^{*}.$$ There  is a non-degenerate symmetric bilinear form on $V$ denoted by $(\cdot, \cdot)_V$, which identifies $O_N(\mathbb{C})$ with $O(V)$.  We can identify $V$ with $V^{*}$ by sending $v \in V$ to the linear functional $(v, \cdot) \in V^{*}$. This is not only the isomorphism of vector spaces but also as $O(V)$-modules since $\forall u \in V$:  
$$(g \rhd (v, \cdot)_V)(u) = (v, g^{-1}(u))_V = (g(v), u) = (g(v), \cdot)_V(u).$$
Therefore $$\tilde{W} = W \oplus W^{*} \underset{O(V)}{\cong} V \otimes (E \oplus E^{*}) = V \otimes \tilde{E}$$
Denote by $(\cdot, \cdot)_{\tilde{W}}$ the bilinear form on $\tilde{W} = W \oplus W^{*}$ defined in \ref{r_1.1}. Similarly, we denote by $(\cdot, \cdot)_{\tilde{E}}$ the respective bilinear form on $\tilde{E}$. 

\begin{proposition}
    \label{prop_4.1}
    Bilinear form $(\cdot, \cdot)_{\tilde{W}}$ on $\tilde{W}$ is the tensor product of the bilinear forms $(\cdot, \cdot)_V$ and $(\cdot, \cdot)_{\tilde{E}}$ on the factors. 
\end{proposition}
\begin{proof}
It is enough to check that the forms $(\cdot, \cdot)_{\tilde{W}}$ and $(\cdot, \cdot)_{V \otimes \tilde{E}}:=(\cdot, \cdot)_V \otimes (\cdot, \cdot)_{\tilde{E}}$ agree on the basis of $\tilde{W}$. 
Let $(e_1, \ldots, e_n)$ be a basis of $E$ and $(e_1^{*}, \ldots e_n^{*})$ be its dual basis in $E^{*}$.  Let $(v_1, \ldots v_N)$ be an orthonormal basis of $V$ with respect to $( \cdot, \cdot)_{V}$. We obtain the basis of $\tilde{W}$ consisting of simple tensors $v_i \otimes e_j$ and $v_i \otimes e_j^{*}$. Tensors of the type $v_i \otimes e_j$ span the subspace $W \subset \tilde{W}$ and tensors of the type $v_i \otimes e_j^{*}$ span $W^{*} \subset \tilde{W}$.
Then 
\begin{gather*}
(v_{i_1} \otimes e_{j_1} + v_{i_2} \otimes e^{*}_{j_2} ,v_{i_3} \otimes e_{j_3} + v_{i_4} \otimes e^{*}_{j_4})_{\tilde{W}} = \langle v_{i_2} \otimes e^{*}_{j_2}, v_{i_3} \otimes e_{j_3} \rangle + \langle v_{i_1} \otimes e_{j_1},  v_{i_4} \otimes e^{*}_{j_4}\rangle = \\
= (v_{i_2}, v_{i_3})_{V} \cdot (e^{*}_{j_2}, e_{j_3})_{\tilde{E}} + (v_{i_1}, v_{i_4})_{V} \cdot ( e_{j_1}, e^{*}_{j_4})_{\tilde{E}} = (v_{i_2}, v_{i_3})_{V} \cdot (e^{*}_{j_2}, e_{j_3})_{\tilde{E}} + (v_{i_1}, v_{i_4})_{V} \cdot ( e_{j_1}, e^{*}_{j_4})_{\tilde{E}} +\\ + (v_{i_1}, v_{i_3})_{V} \cdot (e_{j_1}, e_{j_3})_{\tilde{E}} + (v_{i_2}, v_{i_4})_{V} \cdot ( e^{*}_{j_2}, e^{*}_{j_4})_{\tilde{E}} = (v_{i_1} \otimes e_{j_1} + v_{i_2} \otimes e^{*}_{j_2} ,v_{i_3} \otimes e_{j_3} + v_{i_4} \otimes e^{*}_{j_4})_{V \otimes \tilde{E}}
\end{gather*}
and we are done.
\end{proof}

\begin{corollary}
\label{cor_4.1}
Lie groups $O(\tilde{E})$ and $O(V)$ are naturally embedded into $O(\tilde{W})$. Indeed, we have inclusions 
$$i_{\tilde{E}}: O(\tilde{E}) \hookrightarrow O(\tilde{W})$$ defined by $$f \mapsto Id_V \otimes f,$$ and $$i_{V}: O(V) \hookrightarrow O(\tilde{W})$$ defined by $$g \mapsto g \otimes Id_{\tilde{E}}.$$
\end{corollary}
So, we can treat groups $O(\tilde{E})$ and $O(V)$ as subgroups of $O(\tilde{W})$. 

\begin{proposition}
\label{prop_4.2}
Subgroups $O(\tilde{E})$ and $O(V)$ of the group $O(\tilde{W})$ are centralizers of each other. 
\end{proposition}

\begin{proof}
Any operator from $O(\tilde {W})$ lies in $\text{End}(\tilde{W}) \cong \text{End}(V)\otimes \text{End}(\tilde{E})$. Therefore, we can express any operator $x \in O(\tilde{W})$ in the following form $$ x = f_1 \otimes g_1 + f_2 \otimes g_2 +  \ldots + f_n \otimes g_n$$
Where $f_i$'s and $g_i$'s are linearly independent as elements of $\text{End}(V)$ and $\text{End}(\tilde{E})$ respectively.   Assume $x$ commutes with $f \otimes Id_{\tilde{E}} \in O(\tilde{W})$ for every $f \in O(V)$.   

Then each $f_i$ commutes with any $f \in O(V)$, since $g_i$'s are linearly independent. Hence, each $f_i$ is a scalar operator, so one can write $$ x =  f_1 \otimes g_1 + f_2 \otimes g_2 +  \ldots + f_n \otimes g_n = Id_{V} \otimes g$$
Operator $x$ has to preserve the bilinear form $(\cdot, \cdot)_{\tilde{W}}$. By proposition \ref{prop_4.1} $(\cdot, \cdot)_{\tilde{W}} =(\cdot, \cdot)_V \otimes (\cdot, \cdot)_{\tilde{E}}$, so $g$ must lie in $O(\tilde{E})$. We conclude that the operators from $O(\tilde{W})$, which commute with the operators from $O(V) \subset O(\tilde{W})$ are precisely the operators from $O(\tilde{E}) \subset O(\tilde{W})$. Similarly, we see that the operators from $O(\tilde{W})$, which commute with the operators from $O(\tilde{E}) \subset O(\tilde{W})$ are precisely the operators from $O(V) \subset O(\tilde{W})$. Hence, $O(V)$ and $O(\tilde{E})$ are centralizers of each other inside $O(\tilde{W})$ and we are done.
\end{proof}
\begin{corollary}
\label{cor_4.2}
Lie algebras $\mathfrak{o}(\tilde{E})$ and $\mathfrak{o}(V)$ are naturally embedded into $\mathfrak{o}(\tilde{W})$ by the differentials of the Lie group homomorphisms $i_{\tilde{E}}$ and $i_{V}$ established in corollary \ref{cor_4.1}. As subalgebras of $\mathfrak{o}(\tilde{W})$, $\mathfrak{o}(\tilde{E})$ and $\mathfrak{o}(V)$ centralize each other. 
\end{corollary}
Consider the following map: 
$$ s_{W} \circ d_e i_{\tilde{E}}: \mathfrak{o}(\tilde{E}) \hookrightarrow \mathfrak{o}(\tilde {W}) \longrightarrow End_{\mathbb{C}}(\Lambda W).$$
It defines the $\mathfrak{o}(\tilde{E})$ representation structure on the space $\Lambda W$.  Notice that as a vector space $\Lambda W$ is isomorphic to $\Lambda E ^{\otimes N}$. Lets fix the isomorphism between these spaces:  $$\gamma: \Lambda E^{\otimes N} \longrightarrow \Lambda W$$
defined by
\begin{equation} \label{eq_4.1}
\gamma : \bigotimes_{j = 1}^{N} \Lambda_{r = 1}^{k_j} e_{i_{jr}} \mapsto \Lambda_{j = 1}^{N} \Lambda_{r = 1}^{k_j} (v_j \otimes e_{i_{jr}}),
\end{equation}

where $(v_1, v_2, \ldots v_N)$ is an orthonormal basis of $V$ with respect to $(\cdot, \cdot)_V$ and $(e_1, e_2, \ldots e_n)$ is some basis of the space $E$. 

On the vector space $\Lambda E^{\otimes N}$ there is the structure of $\mathfrak{o}(\tilde{E})$-module given by the $N$-th tensor power of the spinor representation $s_E$ with the base space $E$.   

\begin{proposition}
\label{prop_4.3}
Map $\gamma$ is an isomorphism of representations of the algebra Lie $\mathfrak{o}(\tilde{E})$. Representation structure on the space $\Lambda W$ is given by the homomorphism  $$s_{W} \circ d_e i_{\tilde{E}}: \mathfrak{o}(\tilde{E}) \longrightarrow End_{\mathbb{C}}(\Lambda W).$$
Representation structure on the space $\Lambda E^{\otimes N}$ is given by the homomorphism $$s_E^{\otimes N}: \mathfrak{o}(\tilde{E}) \longrightarrow \text{End}_{\mathbb{C}}(\Lambda E^{\otimes N}).$$
\end{proposition}
\begin{proof}
We choose the bases in $\tilde{E}$, $V$, and $\tilde{W}$ as in the proof of proposition \ref{prop_4.1}. Notice that for all $x \in \mathfrak{o}(\tilde{E})$ 
$$d_ei_{\tilde E}(x) =   P_{v_1} \otimes  x +  P_{v_2} \otimes  x + \ldots + P_{v_N} \otimes  x,$$
where $P_{v_i}$ is the orthogonal projection onto $\langle v_i \rangle_{\mathbb{C}}$.
We denote by $E_{v_{i_1} \otimes e_{j_1}^{\circ}, v_j \otimes e_{j_2}^{\circ}}$ the operator from $\mathfrak{gl}(\tilde W)$ that sends the basis vector $v_j \otimes e_{j_2}^{\circ}$ to the basis vector $v_{i_1} \otimes e_{j_1}^{\circ}$ and sends other basis vectors to $0$. Here $\circ$ is either nothing or $*$.  
Isomorphism $\gamma$ induces isomorphism between the endomorphism spaces 
$$\gamma^{\#}: \text{End}_{\mathbb{C}}(\Lambda W) \longrightarrow \text{End}_{\mathbb{C}}(\Lambda E^{\otimes N})$$
Statement of the proposition is equivalent to the equality: 
$$\gamma^{\# } \circ s_{W} \circ d_e i_{\tilde{E}} = s_E^{\otimes N}$$
It is enough to check it on the basis of $\mathfrak{o}(\tilde{E})$. We will check it for the basis vector $E_{ij} - E_{j+n, i+n} \in \mathfrak{o}(\tilde{E})$, the rest is left to the reader. 

Using the formula for $d_ei_{\tilde{E}}$ above, we get
$$d_ei_{\tilde E}(E_{ij} - E_{j+n, i+n}) = \sum_{k=1}^{N} E_{v_k \otimes e_i  ,  v_k \otimes e_j } - E_{v_k \otimes e_j^{*}  ,  v_k \otimes e_i^{*} } $$
We apply $s_W$ to the result: 
$$s_W \left(\sum_{k=1}^{N} E_{v_k \otimes e_i  ,  v_k \otimes e_j } - E_{v_k \otimes e_j^{*}  ,  v_k \otimes e_i^{*} } \right) = \sum_{k=1}^{N} s_W(E_{v_k \otimes e_i  ,  v_k \otimes e_j } - E_{v_k \otimes e_j^{*}  ,  v_k \otimes e_i^{*} }) = $$$$= \frac{1}{2} \sum_{k=1}^{N} M_{v_k \otimes e_i} \circ D_{(v_k \otimes e_j)^{*}} - D_{(v_k \otimes e_j)^{*}} \circ M_{v_k \otimes e_i}.$$
The latter is due to corollary \ref{cor_1.2}.  To finish the proof we need to observe that: $$\gamma^{\# }\left(\frac{1}{2} \sum_{k=1}^{N} M_{v_k \otimes e_i} \circ D_{(v_k \otimes e_j)^{*}} - D_{(v_k \otimes e_j)^{*}} \circ M_{v_k \otimes e_i}\right) =  \frac{1}{2}(M_{e_i} \circ D_{e_j^{*}} - D_{e_j^{*}} \circ M_{e_i}) \otimes Id_{\Lambda E}^{\otimes N-1} + \ldots$$
$$ \ldots + Id_{\Lambda E} \otimes \frac{1}{2}(M_{e_i} \circ D_{e_j^{*}} - D_{e_j^{*}} \circ M_{e_i}) \otimes Id_{\Lambda E}^{\otimes N-2} + \ldots  $$
$$\ldots + Id_{\Lambda E}^{\otimes N-1}  \otimes \frac{1}{2}(M_{e_i} \circ D_{e_j^{*}} - D_{e_j^{*}} \circ M_{e_i}) = s_E^{\otimes N}(E_{ij} - E_{j+n, i+n}).$$

The latter equality is again a consequence of corollary \ref{cor_1.2} and the definition of the algebra Lie representation structure on the tensor product.  

We obtain that $$\gamma^{\# } \circ s_{W} \circ d_e i_{\tilde{E}}(E_{ij} - E_{j+n, i+n}) =  s_E^{\otimes N}(E_{ij} - E_{j+n, i+n}).$$
This concludes the proof. 
\end{proof}

There is a natural action of the group $GL(W)$ on the space $\tilde{W}$. It is easy to see that operators arising from this action lie in $O(\tilde{W})$. Indeed, if we take an element $g \in GL(W)$ and consider its action on two vectors $w_1 + \xi_1$ and $w_2 + \xi_2$ , where  $w_1, w_2 \in W$ and $\xi_1, \xi_2 \in W^{*}$, then we can see that  
\begin{gather*}
(g.(v_1 + \xi_1), g.(v_2 + \xi_2))_{\tilde {W}} = (g.v_1 + g.\xi_1, g.v_2 + g.\xi_2)_{\tilde{W}} = \\
=  g.\xi_1(g.v_2) +  g.\xi_2(g.v_1) = \xi_1(g^{-1}g.v_2) + \xi_2(g^{-1}g.v_1) = \\
 = \xi_1(v_2) + \xi_2(v_1) = ( v_1 + \xi_1, v_2 + \xi_2)_{\tilde{W}}.
\end{gather*}

So, we get an injective Lie group homomorphism  $$j_W: GL(W) \hookrightarrow O(\tilde{W}).$$
There is also natural injective Lie group homomorphism $$ext_W: GL(W) \hookrightarrow GL(\Lambda W).$$
Recall that $W \cong V \otimes E$, so there is a natural embedding $$\bar{i}_V : O(V) \hookrightarrow GL(W)$$ defined by
$$\bar{i}_V: g \mapsto g \otimes \text{Id}_E.$$
One can see that $$j_W \circ \bar{i}_V = i_V$$
Therefore,  algebra Lie $\mathfrak{o}(V)$ is embedded into $\mathfrak{gl}(W)$ via the homomorphism $d_e\bar{i}_V$ and the embeddings $d_ej_W \circ d_e\bar{i}_V$, $d_ei_V$ of $\mathfrak{o}(V)$ into $\mathfrak{o}(\tilde{W})$ coincide. 

\begin{proposition}
\label{prop_4.4}
Homomorphisms of Lie algebras 
  $$d_e ext_W :   \mathfrak{gl}(W) \hookrightarrow \text{End}_{\mathbb{C}}(\Lambda W)$$
  and $$s_W \circ d_e j_W : \mathfrak{gl}(W) \hookrightarrow \text{End}_{\mathbb{C}} (\Lambda W).$$
  
  coincide on the subalgebra $\mathfrak{o}(V) \subset \mathfrak{gl}(W)$. 
    
\end{proposition}

\begin{proof}
We choose the basis in $W$ consisting of the simple tensors $v_t \otimes e_s$. It is clear that the subalgebra $\mathfrak{o}(V)$ lies inside the subspace spanned by the matrix units $E_{ij} \in \mathfrak{gl}(W)$, such that $i\neq j$.  Hence, it is enough to check that for matrix unit $E_{ij} \in \mathfrak{gl}(W)$, $i\neq j$, the following equality holds: 
$$d_e ext_W(E_{ij}) = s_W \circ d_e j_W(E_{ij}).$$
Let $(w_1, w_2, \ldots, w_{N \cdot n})$ be the chosen basis inside $W$, meaning that each $w_i$ is $v_t \otimes e_s$ for some $t$ and $s$. Consider the action of $d_e ext_W(E_{ij})$ on some the element of $\Lambda W$: 
$$d_e ext_W(E_{ij}) \rhd w_{l_1} \wedge \ldots \wedge w_{l_m} = \sum_{k = 1}^{m} \delta_{l_{k}j} \cdot w_{l_1} \wedge \ldots w_{l_{k-1}} \wedge w_{i} \wedge w_{l_{k+1}} \wedge \ldots \wedge w_{l_m}.$$
In other words, one can say that if there is a factor $w_j$ in the monomial, then it gets replaced with the factor $w_i$. If there is a factor $w_i$ in the monomial or there is no factor $w_j$ in the monomial, then the monomial gets canceled. 

Now we want to show that the action of $s_W \circ d_e j_W(E_{ij})$ on monomials coincides with the action of $d_e ext_W(E_{ij})$. We choose the usual basis in $\tilde{W}$: $(w_1, w_2, \ldots w_{N \cdot n}, w_1^{*}, w_{2}^{*} \ldots w_{N \cdot n}^{*})$. It is easy to see that $$d_ej_W(E_{ij}) = E_{ij} - E_{j + N \cdot n, i + N \cdot n}.$$
Hence, $$s_W \circ d_e j_W(E_{ij}) = \frac{1}{2} (M_{w_i}D_{w_j^{*}} - D_{w_j^{*}}M_{w_i}) = M_{w_i} D_{w_j^{*}}.$$
And just by definition of operators $M_{w_i}$ and $D_{w_j^{*}}$ operator  $M_{w_i} D_{w_j^{*}}$ acts the same way as $d_e ext_W(E_{ij})$ on the monomials $ w_{l_1} \wedge \ldots \wedge w_{l_m}$. But these monomials span all $\Lambda W$, so we are done. 
\end{proof}

We state the fundamental result due to Roger Howe \cite{Howe}. 

\begin{theorem}
\label{th_4.1}
Let $V$ be the standard representation of the orthogonal group $O(V) \cong O_N(\mathbb{C})$ $(\dim V = N)$. Denote by $W = \oplus^{n} V  \underset{O(V)}{\cong}  V \otimes E$, where $E$ is the auxiliary space of dimension $n$. $O(V)$ acts on $W$ on the left side. Consider $\Lambda W$ as the $O(V)$ -module. It is completely reducible. Let 
$$\Lambda W  \underset{O(V)}{\cong} \bigoplus_j  I_j$$ be the decomposition of $\Lambda W$ into isotypic components under the action of $O(V)$. Let $\Gamma$ be the image of $\mathfrak{o}(V)$ in $\text{End}_{\mathbb{C}}(\Lambda W)$ corresponding to this action. It follows from \ref{prop_4.4} that this image coincides with the image of embedding $d_ei_V$ of $\mathfrak{o}(V)$ into $\mathfrak{o}(\tilde{W}) \subset \text{End}_{\mathbb{C}}(\Lambda W)$. So, take the centralizer $\Gamma^{'}$ of  $\Gamma$ in $\mathfrak{o}(\tilde{W})$. Then any isotypic component $I_j$ is irreducible under the joint action of $O(V)$ and $\Gamma^{'}$ and has the following form: 
$$U_j \otimes L_j,$$
where $U_j$ is a finite-dimensional irreducible representation of $O(V)$, and $L_j$ - is an irreducible representation of Lie algebra $\Gamma^{'}$, Moreover, $L_j$ and $U_j$ determine each other, so that the map $U_j \leftrightarrow L_j$ is bijective. 
\end{theorem}

This is a very powerful result and applying it together with corollary\ref{cor_4.2} and proposition \ref{prop_4.3}  to our case we immediately get the following corollary.

\begin{corollary}
\label{cor_4.3}
$\Lambda E ^{\otimes N}$ as a representation of Lie algebra $\mathfrak{o}(\tilde{E})\cong \mathfrak{o}_{2n}(\mathbb{C})$ decomposes into the direct sum
    $$\Lambda E^{\otimes N} = \bigoplus_{\lambda \in \Delta^{N} } U_{\lambda} \otimes L_{\lambda},$$
where $L_\lambda$ is the irreducible representation of the highest weight $\lambda$  and $U_\lambda$ is the multiplicity space of subrepresentation $L_\lambda$ in $\Lambda E ^{\otimes N}$. There is a natural structure of irreducible representation of $O_N$ on each multiplicity space $U_\lambda$ and, moreover, $U_\lambda$ and $L_\lambda$ uniquely determine each other.  Equivalently speaking, multiplicity spaces $U_\lambda$ for different $\lambda \in \Delta^N $ are non-isomorphic as $O_N$-modules.
\end{corollary}

\section {Bijection between short Young diagrams and regular cell diagrams}
\label{sec_5}
  We have shown that the multiplicity spaces $U_\lambda$ for each $\lambda \in \Delta^N $ are non-isomorphic irreducible representations of the group $O_N$. It is known from the theorem $19.19$ in \cite{Fulton} that irreducible representations of $O_N$ are indexed by the Young diagrams such that the sum of the lengths of the first two columns does not exceed $N$. We will refer to such diagrams as \textit{short Young diagrams of height $N$}. As a consequence of the corollary \ref{cor_4.3}, for each $\lambda \in \Delta^N $ we have the corresponding short young diagram $\nu$, such that $U_\lambda  \underset{O_N}{\cong}  R_{\nu}$, where $R_{\nu}$ is the irreducible representation of $O_N$ corresponding to the short young diagram $\nu$.  As we saw in \ref{prop_3.2} there is a bijection between the sets $\Delta^N $ and $\mathfrak{D}(N, n)$. We introduce the following notation:
\begin{notation}
Denote by $SYD(N, n)$ the set of short young diagrams of height $N$ that have at most $n$ columns.
\end{notation}
In this chapter, we are going to prove the following theorem. 

\begin{theorem}
\label{th_5.1}
Consider the following map $$\mathcal{F}: \mathfrak{D}(N, n) \rightarrow SYD(N, n).$$ We define it the following way. If $n$ is even, then $$\mathcal{F} (D(\textbf{l}, \textbf{r})) =  (l_n, l_{n-1}, \ldots, l_1)^{t},$$ if n is odd, then $$\mathcal{F} (D(\textbf{l}, \textbf{r})) =  (r_n, l_{n-1}, \ldots, l_1)^{t}.$$ We claim that map $\mathcal{F}$ is a bijection and moreover the image of the cell diagram $D^{N}_{\lambda}$, corresponding to $\lambda \in \Delta^{N} $, under map $\mathcal{F}$ is the short Young diagram $\nu$, such that multiplicity space $U_\lambda$ as $O_N$-module is isomorphic to $R_{\nu}$. 
\end{theorem}

This theorem explains how to explicitly construct the short Young diagram $\nu$ from the highest weight $\lambda \in \Delta^N$. We provide some concrete examples in order to make it easier to visualize the map $\mathcal{F}$. 

\begin{example}
\label{ex_6.1}
Let  $n = 4, N = 7$ and $ \hspace{1mm} \lambda = (\frac{3}{2}, \frac{1}{2}, \frac{1}{2}, -\frac{1}{2}).$
Then $D_{\lambda}^{N}$ is the regular cell diagram from the example \ref{ex_3.1}. Map $\mathcal{F}$ sends it to the short Young diagram drawn on the right. 
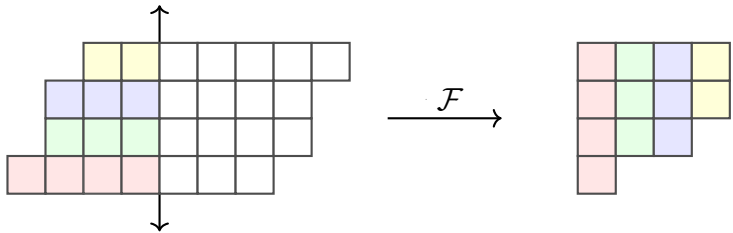
\newline

\centering
\begin{tikzpicture}
\draw[thick,<->] (2,-0.5) -- (2, 2.5) node[anchor=north west] {};
\filldraw[color=black!70, fill=red!10,  thick](0,0) rectangle (0.5,0.5);
\filldraw[color=black!70, fill=red!10,  thick] (0.5,0) rectangle (1,0.5);
\filldraw[color=black!65, fill=red!10,  thick] (1,0) rectangle (1.5,0.5);
\filldraw[color=black!70, fill=red!10,  thick] (1.5,0) rectangle (2,0.5);
\filldraw[color=black!70, fill=white!10,  thick] (2,0) rectangle (2.5,0.5);
\filldraw[color=black!70, fill=white!10,  thick] (2.5,0) rectangle (3,0.5);
\filldraw[color=black!70, fill=white!10,  thick] (3,0) rectangle (3.5,0.5);

\filldraw[color=black!70, fill=green!10,  thick]  (0.5,0.5) rectangle (1,1);
\filldraw[color=black!70, fill=green!10,  thick]  (1,0.5) rectangle (1.5,1);
\filldraw[color=black!70, fill=green!10,  thick]  (1.5,0.5) rectangle (2,1);
\filldraw[color=black!70, fill=white!10,  thick] (2,0.5) rectangle (2.5,1);
\filldraw[color=black!70, fill=white!10,  thick] (2.5,0.5) rectangle (3,1);
\filldraw[color=black!70, fill=white!10,  thick] (3,0.5) rectangle (3.5,1);
\filldraw[color=black!70, fill=white!10,  thick] (3.5,0.5) rectangle (4,1);

\filldraw[color=black!70, fill=blue!10,  thick] (0.5,1) rectangle (1,1.5);
\filldraw[color=black!70, fill=blue!10,  thick] (1,1) rectangle (1.5,1.5);
\filldraw[color=black!70, fill=blue!10,  thick] (1.5,1) rectangle (2,1.5);
\filldraw[color=black!70, fill=white!10,  thick] (2,1) rectangle (2.5,1.5);
\filldraw[color=black!70, fill=white!10,  thick] (2.5,1) rectangle (3,1.5);
\filldraw[color=black!70, fill=white!10,  thick] (3,1) rectangle (3.5,1.5);
\filldraw[color=black!70, fill=white!10,  thick] (3.5,1) rectangle (4,1.5);

\filldraw[color=black!70, fill=yellow!15,  thick] (1,1.5) rectangle (1.5,2);
\filldraw[color=black!70, fill=yellow!15,  thick] (1.5,1.5) rectangle (2,2);
\filldraw[color=black!70, fill=white!10,  thick] (2,1.5) rectangle (2.5,2);
\filldraw[color=black!70, fill=white!10,  thick] (2.5,1.5) rectangle (3,2);
\filldraw[color=black!70, fill=white!10,  thick] (3,1.5) rectangle (3.5,2);
\filldraw[color=black!70, fill=white!10,  thick] (3.5,1.5) rectangle (4,2);
\filldraw[color=black!70, fill=white!10,  thick] (4,1.5) rectangle (4.5,2);

\draw[thick,->] (5,1) -- (6.5,1);
\filldraw[black] (5.5,1.25) circle (0pt) node[anchor=west]{$\mathcal{F}$};

\filldraw[color=black!70, fill=red!10,  thick](7.5,0) rectangle (8,0.5);
\filldraw[color=black!70, fill=red!10,  thick](7.5,0.5) rectangle (8,1);
\filldraw[color=black!70, fill=red!10,  thick](7.5,1) rectangle (8,1.5);
\filldraw[color=black!70, fill=red!10,  thick](7.5,1.5) rectangle (8,2);

\filldraw[color=black!70, fill=green!10,  thick] (8,0.5) rectangle (8.5,1);
\filldraw[color=black!70, fill=green!10,  thick] (8,1) rectangle (8.5,1.5);
\filldraw[color=black!70, fill=green!10,  thick] (8,1.5) rectangle (8.5,2);

\filldraw[color=black!70, fill=blue!10,  thick] (8.5,0.5) rectangle (9,1);
\filldraw[color=black!70, fill=blue!10,  thick] (8.5,1) rectangle (9,1.5);
\filldraw[color=black!70, fill=blue!10,  thick] (8.5,1.5) rectangle (9,2);

\filldraw[color=black!70, fill=yellow!15,  thick] (9,1) rectangle (9.5,1.5);
\filldraw[color=black!70, fill=yellow!15,  thick] (9,1.5) rectangle (9.5,2);

\end{tikzpicture}
\newline
\end{example}
\begin{example}
\label{ex_6.2}
We also present an example for odd $n$. Let $n = 5$, $N = 6$ and $ \hspace{1mm} \lambda = (3,2, 1, 1,-1).$
Then $D_{\lambda}^{N}$ is the regular cell diagram drawn on the left. Map $\mathcal{F}$ sends it to the short Young diagram drawn on the right.  

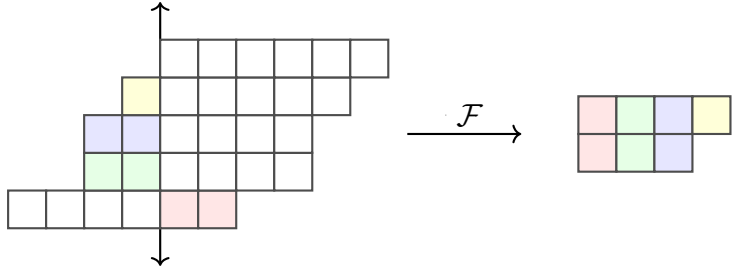
\centering
\begin{tikzpicture}

\draw[thick,<->] (2,-0.5) -- (2, 3) node[anchor=north west] {};
\filldraw[color=black!70, fill=white!10,  thick](0,0) rectangle (0.5,0.5);
\filldraw[color=black!70, fill=white!10,  thick] (0.5,0) rectangle (1,0.5);
\filldraw[color=black!70, fill=white!10,  thick] (1,0) rectangle (1.5,0.5);
\filldraw[color=black!70, fill=white!10,  thick] (1.5,0) rectangle (2,0.5);
\filldraw[color=black!70, fill=red!10,  thick] (2,0) rectangle (2.5,0.5);
\filldraw[color=black!70, fill=red!10,  thick] (2.5,0) rectangle (3,0.5);

\filldraw[color=black!70, fill=green!10,  thick]  (1,0.5) rectangle (1.5,1);
\filldraw[color=black!70, fill=green!10,  thick]  (1.5,0.5) rectangle (2,1);
\filldraw[color=black!70, fill=white!10,  thick] (2,0.5) rectangle (2.5,1);
\filldraw[color=black!70, fill=white!10,  thick] (2.5,0.5) rectangle (3,1);
\filldraw[color=black!70, fill=white!10,  thick] (3,0.5) rectangle (3.5,1);
\filldraw[color=black!70, fill=white!10,  thick] (3.5,0.5) rectangle (4,1);

\filldraw[color=black!70, fill=blue!10,  thick] (1,1) rectangle (1.5,1.5);
\filldraw[color=black!70, fill=blue!10,  thick] (1.5,1) rectangle (2,1.5);
\filldraw[color=black!70, fill=white!10,  thick] (2,1) rectangle (2.5,1.5);
\filldraw[color=black!70, fill=white!10,  thick] (2.5,1) rectangle (3,1.5);
\filldraw[color=black!70, fill=white!10,  thick] (3,1) rectangle (3.5,1.5);
\filldraw[color=black!70, fill=white!10,  thick] (3.5,1) rectangle (4,1.5);

\filldraw[color=black!70, fill=yellow!15,  thick] (1.5,1.5) rectangle (2,2);
\filldraw[color=black!70, fill=white!10,  thick] (2,1.5) rectangle (2.5,2);
\filldraw[color=black!70, fill=white!10,  thick] (2.5,1.5) rectangle (3,2);
\filldraw[color=black!70, fill=white!10,  thick] (3,1.5) rectangle (3.5,2);
\filldraw[color=black!70, fill=white!10,  thick] (3.5,1.5) rectangle (4,2);
\filldraw[color=black!70, fill=white!10,  thick] (4,1.5) rectangle (4.5,2);

\filldraw[color=black!70, fill=white!10,  thick] (2,2) rectangle (2.5,2.5);
\filldraw[color=black!70, fill=white!10,  thick] (2.5,2) rectangle (3,2.5);
\filldraw[color=black!70, fill=white!10,  thick] (3,2) rectangle (3.5,2.5);
\filldraw[color=black!70, fill=white!10,  thick] (3.5,2) rectangle (4,2.5);
\filldraw[color=black!70, fill=white!10,  thick] (4,2) rectangle (4.5,2.5);
\filldraw[color=black!70, fill=white!10,  thick] (4.5,2) rectangle (5,2.5);

\draw[thick,->] (5.25,1.25) -- (6.75,1.25);
\filldraw[black] (5.75,1.5) circle (0pt) node[anchor=west]{$\mathcal{F}$};

\filldraw[color=black!70, fill=red!10,  thick](7.5,0.75) rectangle (8,1.25);
\filldraw[color=black!70, fill=red!10,  thick](7.5,1.25) rectangle (8,1.75);

\filldraw[color=black!70, fill=green!10,  thick](8,0.75) rectangle (8.5,1.25);
\filldraw[color=black!70, fill=green!10,  thick](8,1.25) rectangle (8.5,1.75);

\filldraw[color=black!70, fill=blue!10,  thick](8.5,0.75) rectangle (9,1.25);
\filldraw[color=black!70, fill=blue!10,  thick](8.5,1.25) rectangle (9,1.75);

\filldraw[color=black!70, fill=yellow!15,  thick] (9,1.25) rectangle (9.5,1.75);

\end{tikzpicture}
\end{example}

To prove this theorem, we need to study the structure of the $O_N$ module on multiplicity spaces $U_\lambda$. Our approach will be the following. $O_N$-module structure  naturally defines $\mathfrak{o}_N(\mathbb{C})$-module structure on the multiplicity space $U_\lambda$. We will study $U_\lambda$ as a $\mathfrak{o}_{N}(\mathbb{C})$-module and then once we manage to decompose it into the sum of simple $\mathfrak{o}_{N}(\mathbb{C})$-modules we will be able to say a lot about its $O_N$-module structure. The following two propositions from \cite{Fulton} justify this approach. 

\begin{proposition} (Theorem 19.20 in \cite{Fulton}) 
\label{prop_5.1}

Let $R_\nu$ be the irreducible representation of the group $O_N$ corresponding to a short Young diagram $\nu \hspace{1mm}(\nu^t_1 + \nu_2^t \leqslant N)$, such that $\nu^t_1 \leqslant \frac{N}{2}$. Assume that we fixed the standard choice of Cartan subalgebra inside $\mathfrak{o}_N(\mathbb{C})$ and the natural basis in its dual. All highest weights are written in this basis.
\begin{itemize}
    \item If $N = 2m + 1$ and $\nu = (\nu_1 \geqslant \nu_2  \geqslant \ldots \geqslant \nu_m \geqslant 0)$, then $R_{\nu}$ is an irreducible representation of $\mathfrak{o}_N(\mathbb{C})$ of the highest weight $\nu$.
    \item If $N = 2m$ and $\nu = (\nu_1 \geqslant \nu_2  \geqslant \ldots \geqslant \nu_{m-1} \geqslant 0)$, then $R_{\nu}$ - irreducible representation $\mathfrak{o}_N(\mathbb{C})$ of the highest weight $\nu$.
    \item If $N = 2m$ and $\nu = (\nu_1 \geqslant \nu_2  \geqslant \ldots \geqslant \nu_{m} > 0)$, then $R_{\nu}$  is the sum of two irreducible representations  of  $\mathfrak{o}_N(\mathbb{C})$ with the highest weights $\nu = (\nu_1, \ldots, \nu_m)$ and $\nu^{\sigma} = (\nu_1, \ldots, - \nu_m)$.
\end{itemize}
\end{proposition}

We are yet to cover the case, where $R_{\nu}$ is the irreducible representation of $O_N(\mathbb{C})$ corresponding to a short Young diagram $\nu$, such that $\nu_1^{t} > \frac{N}{2}$. Let us introduce the following definition. 

\begin{definition}
\label{def_5.1}
  Short Young diagram $\nu^{\dagger}$  is called  \textbf{associated } to short Young diagram $\nu$ if $\nu^{\dagger}$ is obtained from $\nu$ by replacing the first column $\nu_1^{t}$ by $N-\nu_1^{t}$. Short Young diagram is called \textbf{self-associated } if  $\nu^{\dagger} = \nu$. 
\end{definition}

This definition makes sense since  $\nu^t_1 + \nu_2^t \leqslant N$,  so $N - \nu_1^{t} \geqslant \nu_2^{t}$ and $N - \nu_1^{t} + \nu_2^{t}  \leqslant N$, since $\nu^{t}_1 \geqslant  \nu^{t}_2$. 

\begin{proposition} (Page 297 of \cite{Fulton}) 

\label{prop_5.2}
    Let $R_\nu$ and $R_{\nu^{\dagger}}$ be two irreducible representations of $O_N$ corresponding to short Young diagrams $\nu$ and $\nu^{\dagger}$. Then $R_\nu$ and $R_{\nu^{\dagger}}$ are isomorphic as representations of $\mathfrak{o}_{N}(\mathbb{C})$. 
\end{proposition}

So, we know how to decompose any irreducible representation of $O_N$ into the sum of irreducible representations of Lie algebra $\mathfrak{o}_{N}(\mathbb{C})$. For the technical reasons, it will be convenient to introduce the following notation. 

\begin{notation}
    Let $\nu$ be the short Young diagram of the height $N$. Then denote by $s(\nu)$ the shorter of  the diagrams $\nu$ and $\nu^{\dagger}$:
    $$s(\nu) = \begin{cases}
    \nu, \hspace{1mm} \text{if} \hspace{1mm} \nu_1^{t} \leqslant \frac{N}{2}, \\
    \nu^{\dagger},\hspace{1mm} \text{otherwise}.
    \end{cases}$$
 Obviously, in case the diagram $\nu$ is self-associated  i.e. $\nu_1^t = \frac{N}{2}$, $s(\nu) = \nu = \nu^{\dagger}$. 
\end{notation}
In chapter \ref{sec_4} we defined both $O(V)$ and $\mathfrak{o}(\tilde{E})$ representation structure on the space $\Lambda W$.  We can treat multiplicity space $U_\lambda$ as the subspace of the singular vectors of the weight $\lambda$ in $\Lambda W$. $O(V)$-module structure on $U_\lambda$ is inherited from the $O(V)$-module structure on $\Lambda W$. Indeed, since the action of $O(V)$ commutes with the action of $\mathfrak{o}(\tilde{E})$ on $\Lambda W$  (see \ref{th_5.1}), then a subspace of the singular vectors of a fixed weight is a $O(V)$ subrepresentation. Similarly, $U_\lambda$ is a $\mathfrak{o}(V)$-submodule of $\Lambda W$.  We know that the images of $\mathfrak{o}(V)$ and $\mathfrak{o}(\tilde{E})$ in $\text{End}_{\mathbb{C}}(\Lambda W)$ commute since they as subalgebras of $\mathfrak{o}(\tilde{W}) \subset \text{End}_{\mathbb{C}}(\Lambda W)$ centralize each other (see \ref{cor_1.1} and \ref{cor_4.2}). Therefore, $\Lambda W$ is a finite-dimensional representation of the Lie algebra $\mathfrak{o}(V) \oplus \mathfrak{o}(\tilde{E})$.  Lie algebra  $\mathfrak{o}(V) \oplus \mathfrak{o}(\tilde{E})$ is semisimple, so in order to study its finite-dimensional representations, we need to fix the choice of the Cartan subalgebra in $\mathfrak{o}(V) \oplus \mathfrak{o}(\tilde{E})$ and polarization of the root system. It is natural to choose Cartan subalgebra to be the direct sum of Cartan subalgebras of the summands. So, we need to say a few words about the summands. 

We have two different situations depending on the parity of the dimension of $V$.  
In case dimension of $V$ is even, $\dim(V) = N = 2d$ we choose a basis $$\textbf{v}= (v_1, v_2, \ldots, v_N)$$
in $V$, such that the Gram matrix of the symmetric bilinear form on $V$ in this basis equals 
$$G = \begin{bmatrix}
    0 & I_{d} \\
    I_{d} & 0
\end{bmatrix}$$
Then $\mathfrak{o}(V)$ as matrix subalgebra of $\mathfrak{gl}(V)$ written in the basis $\textbf{v}$ is essentially the same matrix Lie algebra as the one we observed in \ref{r_1.2}.  

We choose Cartan subalgebra of $\mathfrak{o}(V)$ to be the subalgebra of diagonal matrices in $\mathfrak{o}(V)$: 
$$\mathfrak{h}_{\mathfrak{o}(V)} = span_{\mathbb{C}}\{t_i = E_{ii} - E_{i+d, i+d} \hspace{1mm} | \hspace{1mm} 1 \leqslant i \leqslant d \}.$$
We also fix the basis in $\mathfrak{h}_{\mathfrak{o}(V)}$: $$\textbf{t} = (t_i)_{1\leqslant i \leqslant d}$$ and its dual in $\mathfrak{h}^{*}_{\mathfrak{o}(V)}$: $$\textbf{t}^{*} = (t_i^{*})_{1\leqslant i \leqslant d}.$$

Root system in basis $\textbf{t}^{*}$ has the following form: $$R = \{ \pm t_i^* \pm t_j^*\ | \hspace{1mm} i \neq j \}.$$
We choose polarization so that 
$$R_+ = \{t_i^* + t_j^* \hspace{1mm} | \hspace{1mm} i \neq j \} \cup \{ t_i^* - t_j^* \hspace{1mm}| \hspace{1mm} i < j \}.$$
Then $\mathfrak{o}(V)$ admits triangular decomposition $$\mathfrak{o}(V) = \mathfrak{n}_{-}(\mathfrak{o}(V)) \oplus \mathfrak{h}_{\mathfrak{o}(V)} \oplus \mathfrak{n}_{+}(\mathfrak{o}(V)),$$ where 
$$\mathfrak{n}_{-}(\mathfrak{o}(V)) = \text{span}_{\mathbb{C}}\{ E_{i + d, j} - E_{j + d, i} \hspace{1mm} | \hspace{1mm} i \neq j \} \oplus 
\text{span}_{\mathbb{C}} \{E_{ij} - E_{j+d, i+d} \hspace{1mm} | \hspace{1mm} i > j \}$$
and 
$$\mathfrak{n}_{+}(\mathfrak{o}(V)) = \text{span}_{\mathbb{C}}\{ E_{i, j + d} - E_{j, i + d} \hspace{1mm} | \hspace{1mm} i \neq j \} \oplus \text{span}_{\mathbb{C}}\{ E_{ij} - E_{j+d, i+d} \hspace{1mm} | \hspace{1mm}i < j \}.$$

In case dimension of $V$ is odd, $\dim(V) = N = 2d + 1$, we choose the basis $\textbf{v} = (v_1, v_2, \ldots, v_N)$ in $V$, such that Gram matrix of the symmetric bilinear form on $V$ in this basis equals
$$\begin{bmatrix}
    0 & I_{d} & 0 \\
    I_{d} & 0 & 0 \\
    0 & 0 & 1
\end{bmatrix}$$
Then $\mathfrak{o}(V)$ as matrix subalgebra of $\mathfrak{gl}(V)$ written in the basis $\textbf{v}$ equals $$\mathfrak{o}(V) = \{ x \in \mathfrak{gl}(V) \hspace{1mm}| \hspace{1mm} Gx + x^{t}G = 0 \}.$$
So, in the chosen basis elements of $\mathfrak{o}(V)$ are written as matrices   
$$\begin{bmatrix}
    A & B & e \\
    C & D & f \\
    h & l & 0
\end{bmatrix}$$
such that $A + D^{t} = 0$, $B + B^{t} = 0$, $C + C^{t}$, $e + l^{t} = 0$, $f + h^{t}= 0$.

As in the even case it is convenient to choose the Cartan subalgebra of $\mathfrak{o}(V)$ to be the subalgebra of the diagonal matrices:  
$$\mathfrak{h}_{\mathfrak{o}(V)} = span_{\mathbb{C}}\{ t_i := E_{ii} - E_{i+d, i+d} \hspace{1mm} | \hspace{1mm} 1 \leqslant i \leqslant d \}.$$
The root system in the odd case is a bit larger. In the odd case root system written in the basis $\textbf{t}^{*}$  has the following form: 
$$R = \{ \pm t_i^{*} \pm t_j^{*} \hspace{1mm}| \hspace{1mm} i \neq j \} \cup \{  \pm t_i^{*} \hspace{1mm} | \hspace{1mm} 1 \leqslant i \leqslant d\}.$$
As we can see we get additional roots $\pm t_i$  in the odd case. Let us choose the polarization of the root system, so that
$$R_+ = \{t_i^{*} - t_j^{*} \hspace{1mm} |  \hspace{1mm}  i < j \} \cup \{t_i^{*} + t_j^{*} \hspace{1mm} | \hspace{1mm} i \neq j\} \cup  \{  t_i^{*}\hspace {1mm} | 1 \leqslant i \leqslant d \} $$
Then we get the following triangular decomposition of Lie algebra $\mathfrak{o}(V)$: $$\mathfrak{o}(V) = \mathfrak{n}_{-}(\mathfrak{o}(V)) \oplus \mathfrak{h}_{\mathfrak{o}(V)} \oplus \mathfrak{n}_{+}(\mathfrak{o}(V)).$$
Explicitly:
\begin{gather*}
\mathfrak{n}_{-}(\mathfrak{o}(V)) = \text{span}_{\mathbb{C}}\{ E_{i + d, j} - E_{j + d, i} \hspace{1mm} | \hspace{1mm} i \neq j \} \oplus   \text{span}_{\mathbb{C}}\{ E_{ij} - E_{j+d, i+d} \hspace{1mm} | \hspace{1mm} i > j \} \oplus \\ \oplus \hspace{1mm} \text{span}_{\mathbb{C}}\{ E_{n + i, 2d+1} - E_{2d+1,i}  \hspace{1mm} | \hspace{1mm} 1 \leqslant i \leqslant d \} \end{gather*} 
and
\begin{gather*} 
\mathfrak{n}_{+}(\mathfrak{o}(V)) = \text{span}_{\mathbb{C}}\{ E_{i , j+ d} - E_{j , i+ d} \hspace{1mm} | \hspace{1mm} i \neq j \} \oplus  \text{span}_{\mathbb{C}}\{ E_{ij} - E_{j+d, i+d} \hspace{1mm} | \hspace{1mm} i < j \} \oplus \\ 
\oplus \hspace{1mm}  \text{span}_{\mathbb{C}} \{E_{i, 2d+1} - E_{2d+1, n+i} \hspace{1mm} | \hspace{1mm} 1 \leqslant i \leqslant d \}.
\end{gather*}

We deal with Lie algebra $\mathfrak{o}(\tilde{E})$ the same way we dealt with even case of $\mathfrak{o}(V)$. We have already listed all the necessary results in paragraph 2.  We choose the basis $\tilde{\textbf{e}} =  (e_1, \ldots e_n, e_1^{*}, \ldots e_n^{*})$ in $\tilde{E}$. The choice of Cartan subalgebra in $\mathfrak{o}(\tilde{E})$ is the following:
$$\mathfrak{h}_{\mathfrak{o}(\tilde{E})} = \text{span}_{\mathbb{C}}\{ h_i := E_{ii} - E_{i+n, i+n} \hspace{1mm} | \hspace{1mm} 1 \leqslant i \leqslant n \}$$
Other two summands of the triangular decomposition: 
$$\mathfrak{n}_{-}(\mathfrak{o}(\tilde{E})) = \text{span}_{\mathbb{C}}\{ E_{i + n, j} - E_{j + n, i} \hspace{1mm} | \hspace{1mm} i \neq j \} \oplus 
\text{span}_{\mathbb{C}} \{E_{ij} - E_{j+n, i+n} \hspace{1mm} | \hspace{1mm} i > j \},$$
$$\mathfrak{n}_{+}(\mathfrak{o}(\tilde{E})) = \text{span}_{\mathbb{C}}\{ E_{i, j + n} - E_{j, i + n} \hspace{1mm} | \hspace{1mm} i \neq j \} \oplus \text{span}_{\mathbb{C}}\{ E_{ij} - E_{j+n, i+n} \hspace{1mm} | \hspace{1mm}i < j \}.$$
As we have already announced Cartan subalgebra of the Lie algebra $\mathfrak{o}(V) \oplus \mathfrak{o}(\tilde{E})$ is $$\mathfrak{h}_{\mathfrak{o}(V) \oplus \mathfrak{o}(\tilde{E})} = \mathfrak{h}_{\mathfrak{o}(\tilde{E})} \oplus \mathfrak{h}_{\mathfrak{o}(V)} = \text{span}_{\mathbb{C}}\{h_i \hspace{1mm} | \hspace{1mm}  1 \leqslant i \leqslant n \} \oplus \text{span}_{\mathbb{C}} \{t_j \hspace{1mm}| \hspace{1mm}  1 \leqslant j \leqslant d \} $$
Similar with the other summands of the triangular decomposition of Lie algebra $\mathfrak{o}(V) \oplus \mathfrak{o}(\tilde{E})$:
$$\mathfrak{n}_{+}(\mathfrak{o}(V) \oplus \mathfrak{o}(\tilde{E})) =  \mathfrak{n}_{+}(\mathfrak{o}(V)) \oplus  
 \mathfrak{n}_{+}(\mathfrak{o}(\tilde{E}))$$$$\mathfrak{n}_{-}(\mathfrak{o}(V) \oplus \mathfrak{o}(\tilde{E})) =  \mathfrak{n}_{-}(\mathfrak{o}(V)) \oplus  
 \mathfrak{n}_{-}(\mathfrak{o}(\tilde{E}))$$
Let us choose the dual basis $(\mathbf{t}^*, \mathbf{h}^*)$ in $\mathfrak{h}^*_{\mathfrak{o}(V) \oplus \mathfrak{o}(\tilde{E})}$ so that 
\begin{gather}
h^{*}_i(h_j) = \delta_{ij} \\
h^{*}_i(t_j) = 0 \\
t^{*}_i(h_j) = 0\\
t^{*}_i(t_j) = \delta_{ij} 
\end{gather}

We will write the highest weights of irreducible representations of $\mathfrak{o}(V) \oplus \mathfrak{o}(\tilde{E})$ in this basis. Now we are ready to prove the following important result. 

\begin{theorem}
  \label{th_5.2}
In case dimension of $V$ is even, $\Lambda W$ as representation of $\mathfrak{o}(V) \oplus \mathfrak{o}(\tilde{E})$ decomposes into the sum of irreducible representations 
\begin{equation} \label{eq_5.1}
    \Lambda W \underset{\mathfrak{o}(V) \oplus \mathfrak{o}(\tilde{E})}{\cong}\bigoplus_{\{\lambda \in \Delta^N  \hspace{1mm}| \hspace{1mm} \lambda_n \neq 0\}} L_{(\kappa(\lambda), \lambda)} \oplus \bigoplus_{\{\lambda \in \Delta^N  \hspace{1mm}| \hspace{1mm} \lambda_n = 0\}} L_{(\kappa(\lambda), \lambda)} \oplus  L_{(\kappa(\lambda)^{\sigma}, \lambda)} ,
\end{equation}

where $\kappa(\lambda) = (\text{min} (l_{n}, r_{n}), l_{n-1},\ldots, l_{1})^t$ and $(l_i)_{1 \leqslant i \leqslant n} = \textbf{l}(\lambda, N)$, $(r_i)_{1 \leqslant i \leqslant n} = \textbf{r}(\lambda, N)$ defined by equations \ref{eq_3.1}.  

In case dimension of $V$ is odd, $\Lambda W$ as representation of $\mathfrak{o}(V) \oplus \mathfrak{o}(\tilde{E})$ decomposes into the sum of irreducible representations 
\begin{equation} \label{eq_5.2}
    \Lambda W \underset{\mathfrak{o}(V) \oplus \mathfrak{o}(\tilde{E})}{\cong} \bigoplus_{\lambda \in \Delta^N } L_{(\kappa(\lambda), \lambda)}, 
\end{equation}
where $\kappa(\lambda) = (\text{min} (l_{n}, r_{n}), l_{n-1},\ldots, l_{1})^t$ and $(l_i)_{1 \leqslant i \leqslant n} = \textbf{l}(\lambda, N)$, $(r_i)_{1 \leqslant i \leqslant n} = \textbf{r}(\lambda, N)$ defined by equations \ref{eq_3.1}.  
\end{theorem}
\begin{proof}

Let us decompose $\Lambda W$ into the direct sum of isotypic components under the action of $\mathfrak{o}(\tilde{E})$: 
\begin{gather*}
\Lambda W \underset{\mathfrak{o}(\tilde{E})}{\cong} \bigoplus_{\lambda \in \Delta^N }U_{\lambda} \otimes L_{\lambda}
\end{gather*}
We have already observed that the actions of  $\mathfrak{o}(\tilde{E})$ and $\mathfrak{o}(V)$ on $\Lambda W$ commute,  so the multiplicity spaces $U_\lambda$ realized as subspaces of $\Lambda W$ are naturally its $\mathfrak{o}(V)$-submodules. From corollary \ref{cor_4.3} we know that $U_\lambda$ is an irreducible $O(V)$ subrepresentation of $\Lambda W$.  
So, $U_\lambda$ as a representation of the group $O(V)$ is isomorphic to some $R_{\nu}$, where $\nu$ is the corresponding short Young diagram. Then from the propositions \ref{prop_5.1} and \ref{prop_5.2}, we know that as a representation of  $\mathfrak{o}(V)$, $U_\lambda$ is either an irreducible representation, in case $\nu$ is not self-associated, or is the direct sum of two irreducible representations, in case $\nu$ is self-associated. So, as a representation of $\mathfrak{o}(V) \oplus \mathfrak{o}(\tilde{E})$ each subspace $U_{\lambda} \otimes L_{\lambda}$ is either irreducible or is the sum of two irreducibles. More precisely, $$U_{\lambda} \otimes L_{\lambda} \underset{\mathfrak{o}(V) \oplus \mathfrak{o}(\tilde{E})}{\cong} L_{(s(\nu), \lambda)} $$ or $$U_{\lambda} \otimes L_{\lambda} \underset{\mathfrak{o}(V) \oplus \mathfrak{o}(\tilde{E})}{\cong}L_{(\nu, \lambda)} \oplus L_{(\nu^{\sigma}, \lambda)} .$$In order to determine $\nu$ we need to find highest weight vector with the right side of the weight equal to $\lambda$. Then, if the last coordinate of the left side of the weight equals $0$ it is the situation, where $\nu$ is not a self-associated short Young diagram. If it is not $0$, then $\nu$ is a self-associated short Young diagram and there is another highest weight vector with the right side of the weight equal to $\lambda$ and left side of the weight equal to $\nu^{\sigma}$. When the dimension of $V$ is odd there are no self-associated short Young diagrams and the second case never happens. 
We will consider different cases of parity of the dimension of $V$ separately. Firstly, consider the case when the dimension of $V$ is even. It turns out that it is very easy to explicitly construct the highest weight vector with the right part of the weight equal to $\lambda$: 
    \begin{equation} \label{eq_5.3}
    \xi_\lambda = \Lambda_{i = 1}^{n}(\Lambda_{k = 1}^{\text{min}(r_i, d)} (v_k \otimes e_i) \wedge \Lambda_{k = 1}^{\text{max}(0, r_i - d)} (v_{N - k + 1} \otimes e_i)).
    \end{equation}
We will see later that the same formula works in case of the odd dimension of $V$. 
One may read this formula as follows. We order the basis in $V$ the following way:  $$v_1, \ldots v_d, v_{2d}, v_{2d-1}, \ldots, v_{d+1}.$$ Then, for any $i$ we add $r_i$ factors starting from $v_1 \otimes e_i$. To verify that $
\xi_\lambda$ is the highest weight vector we need to show that it is singular and that it is a weight vector with the right part of the weight equal to $\lambda$. So, we need to show that $$\mathfrak{n}_+(\mathfrak{o}(V) \oplus \mathfrak{o}(\tilde{E})) \rhd \xi_{\lambda} = 0,$$
$$h_i \rhd \xi_\lambda =  \lambda_i \xi_{\lambda}$$
and 
$$t_i \rhd \xi_\lambda = \tau_i \xi_\lambda,$$  
 where $\tau_i$ is some number.

It is easy to describe the action of an element of $\mathfrak{o}(V)$ on an element of $\Lambda W$. 
$$x \rhd (v_{i_1}\otimes e_{i_1}) \wedge  (v_{i_2}\otimes e_{i_2}) \wedge \ldots \wedge (v_{i_k} \otimes e_{i_k}) = \sum_{j = 1}^{k} (v_{i_1}\otimes e_{i_1}) \wedge  (v_{i_2}\otimes e_{i_2}) \wedge \ldots\wedge (x.v_{i_j} \otimes e_{i_j}) \wedge\ldots \wedge (v_{i_k} \otimes e_{i_k}).$$
A little bit more complicated task is to describe an action of an element of $\mathfrak{o}(\tilde{E})$ on $\Lambda W$. As we know, it comes from the homomorphism 
$$ \psi := s_W\circ d_ei_{\tilde E}: \mathfrak{o}(\tilde{E}) \hookrightarrow \mathfrak{o}(\tilde{W}) \hookrightarrow \text{End}_{\mathbb{C}}(\Lambda W).$$In order to compute this action we need to realize an element of $\mathfrak{o}(\tilde{E})$ as an element of $\mathfrak{o}(\tilde{W})$. If we order basis in $\tilde{W}$ the following way: $$v_1 \otimes e_1,\ldots, v_1 \otimes e_n^{*}, v_2 \otimes e_1, \ldots, v_2 \otimes e_n^{*}, \ldots v_N \otimes e_n^{*}$$Then an element $y$ of $\mathfrak{o}(\tilde{E})$ is represented as an element of $\mathfrak{o}(\tilde{W})$ as a $N$-block diagonal matrix with the same blocks and each block is equal to $y$.  So, let $y$ be equal to $E_{ij} - E_{j+n,i+n}$. Then $$\psi(y) = \frac{1}{2} \sum_{k =1}^{N} M_{v_k \otimes e_i} \circ D_{(v_k \otimes e_j)^{*}} - D_{(v_k \otimes e_j)^{*}} \circ M_{v_k \otimes e_i}.$$ If $y$ is equal to $E_{i,j+n} - E_{j,i+n}$, then $$\psi(y) = \sum_{k =1}^{N} M_{v_k \otimes e_i} \circ M_{v_{\bar{k}} \otimes e_j},$$ where 
$$v_{\bar{k}} = \begin{cases} v_{k+d} \hspace{1mm}, \hspace{1mm} \text{if} \hspace{1mm} k \leqslant d, \\ v_{k-d}, \hspace{1mm} \text{if } k > d.\end{cases}$$

Now let's check that $\xi_\lambda$ is the highest weight vector and that the right part of its weight equals $\lambda$. Firstly, let's check that it is singular. Let $y = E_{ij} -E_{j+n,i+n} \in \mathfrak{n}_{+}(\mathfrak{o}(\tilde{E})),$ such that $i < j$.  Then $$y \rhd \xi_\lambda = \sum_{k =1}^{N} M_{v_k \otimes e_i} \circ D_{(v_k \otimes e_j)^*} (\xi_\lambda).$$
Consider each summand separately. If $\xi_\lambda$ has a factor $v_k \otimes e_j$, then it definitely has a factor $v_k \otimes e_i$,  since $r_i \geqslant r_j$.  Therefore, $$M_{v_k \otimes e_i} \circ D_{(v_k \otimes e_j)^*} (\xi_\lambda) = 0.$$If $\xi_\lambda$ has no factor $v_k \otimes e_j$, then obviously $$M_{v_k \otimes e_i} \circ D_{(v_k \otimes e_j)^*} (\xi_{\lambda}) = 0.$$
So each summand equals $0$, meaning that $y \rhd \xi_\lambda = 0$.  Now let $y = E_{i,j+n} - E_{j,i+n} \in \mathfrak{n}_{+}(\mathfrak{o}(\tilde{E}))$. Then $$y \rhd \xi_{\lambda} =  \sum_{k =1}^{N} M_{v_k \otimes e_i} \circ M_{v_{\bar{k}} \otimes e_j} (\xi_{\lambda}).$$
Again consider each summand separately. Either $k$ or $\bar{k}$ is not greater than $\frac{N}{2} = d$. Without loss of generality assume that $k \leqslant \frac{N}{2} $. Since $r_i$ is at least $\frac{N}{2}$ for all $i \leqslant n-1$, the only chance for the summand not to be equal to $0$ is if $i = n$ . Assuming $\xi_\lambda$ doesn't have a factor $v_k \otimes e_n$, we get  $r_n < k$ and it follows that $$r_j \geqslant r_{n-1} \geqslant l_n > N-k$$and, therefore, there is a factor $v_{k+d} \otimes e_j$ and $M_{(v_{\bar{k}} \otimes e_j)} \xi_\lambda = 0$. 
Now let $x = E_{ij} - E_{j+d,i+d} \in \mathfrak{n}_{+}(\mathfrak{o}(V))$. Consider $x \rhd \xi_\lambda$. It is enough to only consider the action of $x$ on the factors of $\xi_\lambda$ of the form $v_j \otimes *$ and $v_{i+d} \otimes *$ $\hspace{1mm}$. On other factors, it acts by 0.  Since $i < j \leqslant d$ ,  in case  $\xi_\lambda$ has a factor $v_j \otimes e_l$ it must also have a factor $v_i \otimes e_l$.  Then $x$ sends the factor $v_j \otimes e_l$ to $v_i \otimes e_l$ and, therefore, $\xi_\lambda$ vanishes.  Similarly, if  $\xi_\lambda$ has a factor $v_{i+d} \otimes e_l$, then it has a factor $v_{j+d} \otimes e_l$. Action of $x$ on the factor $v_{i+d} \otimes e_l$ sends it to  $-v_{j+d} \otimes e_l$, so $\xi_\lambda$ vanishes. 
Finally, let $x = E_{i,j+d} - E_{j,i+d} \in \mathfrak{n}_{+}(\mathfrak{o}(V))$. Consider $x \rhd \xi_\lambda$.  Again, it is enough to only consider the action of $x$ on the factors of $\xi_\lambda$ of the form $v_{j+d} \otimes *$ and $v_{i+d} \otimes *$ $\hspace{1mm}$. It is obvious that if $\xi_\lambda$ has a factor $v_{i+d} \otimes e_l$ or $v_{j+d} \otimes e_l$, then it definitely has both $v_{i} \otimes e_l$ and $v_{j} \otimes e_l$ as factors. So, $x$ acts on $\xi_\lambda$ by $0$. We conclude that $\xi_\lambda$ is singular. 
Let's prove that $\xi_\lambda$ is a weight vector and that the right part of the weight is equal to $\lambda$. We need to show that $\xi_\lambda$ is an eigenvector for any $z \in \mathfrak{h}_{\mathfrak{o}(V) \oplus \mathfrak{o}(\tilde{E})}$. If $z = h_i$, then $$z \rhd \xi_\lambda = \sum_{k =1}^{N} \frac{1}{2}( M_{v_k \otimes e_i} \circ D_{(v_k \otimes e_i)^*} - D_{(v_k \otimes e_i)^*} \circ  M_{v_k \otimes e_i}) \xi_\lambda = \frac{1}{2}(r_i - l_i) \xi_\lambda = \lambda_i \xi_\lambda.$$ If $z = t_i$, then $$z \rhd \xi_\lambda = (\#\{\text{factors of the form} \hspace{2mm} v_i \otimes * \} - \#\{\text{factors of the form} \hspace{2mm} v_{i+d} \otimes * \}) \xi_\lambda.$$Hence, $\xi_\lambda$ is a highest weight vector with right part of the weight equal to $\lambda$.  The next step is to realize the left part of the weight of $\xi_\lambda$ as some function of $\lambda$. One can see that the eigenvalue of the operator $t_i$ is exactly the amount of cells in the $i$-th column to the left of the axis of the cell diagram, corresponding to $\lambda$ , in case $\lambda_n \geqslant 0$ , and to $\lambda^{\sigma}$, in case $\lambda_n < 0$. 
Thus, the left part of the weight of $\xi_\lambda$  is equal to $\kappa(\lambda)$, where $$\kappa(\lambda) = (\text{min} (l_{n}, r_{n}), l_{n-1},\ldots, l_{1})^{t}.$$
Now we need to figure out when $\kappa(\lambda)_d$ is not equal to $0$ (or in other words when do we have two highest weight vectors with the right part of the weight equal to $\lambda$). In order for $\kappa(\lambda)_d$ not to be equal to $0$ we need to have at least $1$ cell in the $d$-th row of the Young diagram $\kappa(\lambda)$ or equivalently at least $1$ cell in the $d$-th column of $$\mu = (\text{min}(l_{n}, r_{n}), l_{n-1},\ldots, l_{1}).$$ So, $\text{min}(l_{n}, r_{n})$ has to be at least $d$. But $l_n +r_n = N = 2d$. Therefore, $l_n = r_n = d$ or equivalently $\lambda_n = 0$. 

This concludes the proof of the decomposition formula \ref{eq_5.1} for even dimension case. I want to add a little remark. In case $\lambda_n = 0$, it is not hard to construct the second highest weight vector with the weight $(\kappa(\lambda)^{\sigma}, \lambda)$. All you need is to interchange in $\xi_\lambda$ vectors $v_d$ with vectors $v_{2d}$. Clearly, it only affects the $d$-th coordinate of the left part of the weight by changing its sign. It remains to check that the vector we obtained is singular. This is a straightforward calculation similar to what we have done with $\xi_\lambda$. 
The odd dimension case is similar. The same formula \ref{eq_5.3} gives us $\xi_\lambda$ -  the highest weight vector with the right part of the weight equal to $\lambda$. Similarly to an even dimension case, one can verify that the left part of its weight equals $\kappa(\lambda)$. But in the odd case, the subalgebra $\mathfrak{n}_+(\mathfrak{o}(V))$ is slightly bigger. So, we need to check that $\xi_\lambda$ also vanishes under the action of these "new" operators. It is easy to check directly and is left as an exercise for the reader. As mentioned before in case where the dimension of $V$ is odd all multiplicity spaces $U_\lambda$ are irreducible representations of algebra Lie $\mathfrak{o}(V)$, since there are no self-associated short Young diagrams in the odd dimension case. Thus, we get the decomposition stated by the formula \ref{eq_5.2}. 
 \end{proof}

The theorem we have just proved gives us the decomposition of the multiplicity spaces $U_\lambda$ into the direct sum of irreducible representations of the algebra Lie $\mathfrak{o}(V)$. It doesn't give us the full answer to our initial question about the isomorphism class of $U_\lambda$ as a representation of the group $O(V)$. We got the complete answer only when $\lambda_n = 0$. In this case we know from the theorem \ref{th_5.2} that $U_\lambda \underset{\mathfrak{o}(V)}{\cong} L_{\kappa(\lambda)} \oplus  L_{\kappa(\lambda)^{\sigma}}$. Using proposition \ref{prop_5.1}, we get that $U_\lambda \cong R_{\kappa(\lambda)}$  as a representation of $O(V)$. But if $\lambda_n \neq 0$, then we don't get the answer to our question straight away, since $U_\lambda$ as a representation of $O(V)$ can still be isomorphic to either $R_{\kappa(\lambda)}$ or $R_{\kappa(\lambda)^{\dagger}}$, because as representations of $\mathfrak{o}(V)$  they are both isomorphic to $L_{\kappa(\lambda)}$. To determine the isomorphism class of $U_\lambda$ in case $\lambda_n \neq  0$ we need the following facts.

\begin{fact} \label{f_5.1}
Let $N = 2d + 1$. Consider an irreducible representation $U_\nu$ of the group $O_N$ corresponding to the short Young diagram $\nu$. Then the element $-\text{Id} \in O_N$ acts on $U_\nu$ by $(-1)^{|\nu|}$. 
\end{fact}
\begin{fact}
 \label{f_5.2}
 Let $N = 2d$. Consider irreducible representation $U_\nu$ of the group $O_N$ corresponding to non self-associated  short Young diagram $\nu$ $(\nu \neq \nu^{\dagger})$.  Let $v_1, v_2, \ldots v_d, v_{d+1}, \ldots v_{2d}$ be the basis in which gram matrix $G = \begin{bmatrix}
    0 & I_{d} \\
    I_{d} & 0
\end{bmatrix}$ . Consider the element $g_d $ of an orthogonal group that permutes vectors $v_d$ and $v_{2d}$ and acts by identity operator on the other vectors from the basis. Then $g_d$ acts on $U_\nu$ by $1$ if $s(\nu) = \nu$ and by $-1$ if $s(\nu) = \nu^{\dagger}$. 
\end{fact}
The proofs of these facts can be found in \cite{Goodman}.

Using these facts we will prove the result that was announced at the start of the chapter. The rest of this chapter is dedicated to the proof of the theorem \ref{th_5.1}

\begin{proof} 
Let us deal with the case when $N = 2d$ first. We want to understand the isomorphism class of multiplicity space $U_\lambda$ for each $\lambda \in \Delta^N $  as a representation of $O_N$. 
If $\lambda_n = 0$, then we already know that $U_\lambda$ is isomorphic to $R_{\kappa(\lambda)}$. In this particular case $$r_n = l_n = \frac{N}{2} = d,$$so no matter the parity of $n$ we have $\mathcal{F}(\lambda) = \kappa(\lambda)$ and $U_\lambda \cong R_{\mathcal{F}(\lambda)}$. 
Now consider the case $\lambda_n \neq 0$. We know that $U_\lambda \otimes L_\lambda$ is isomorphic to $U_\lambda^{\oplus \dim L_\lambda}$ as a representation of $O_N$. Consider the action of $g_d \in O_N$ on the vector $\xi_\lambda \in U_\lambda \otimes L_\lambda$  defined in the proof of the theorem \ref{th_5.2} by the formula \ref{eq_5.3}. Since $\lambda_n \neq 0$, the factors of $\xi_\lambda$ of the form $v_d \otimes *$  are in one-to-one correspondence with the factors of the form $v_{2d} \otimes *$. In other words, if we have the factor $v_d  \otimes e_i$ in $\xi_\lambda$, then we also have factor $v_{2d} \otimes e_i$ and vice versa. The action of $g_d$ on $\xi_\lambda$ permutes the factors  $v_d \otimes e_i$ with $v_{2d}  \otimes e_i$ for every $i$. It is easy to see then that $$g_d \rhd \xi_\lambda = (-1)^{\#\{\text{factors of the form} \hspace{1mm} v_d \otimes e_i\}}.$$It is clear from the construction of $\xi_\lambda$ that the amount of factors of $v_d \otimes e_i$ in $\xi_\lambda$ is either $n$ or $n-1$. The latter happens if and only if  $\lambda_n < 0.$ 
Now assume first that $n$ is even. 
In case $\lambda_n > 0$ we have that $$g_d \rhd \xi_\lambda = (-1)^{n} \xi_\lambda = \xi_\lambda,$$which means that $g_d$ acts by identity map on $U_\lambda$. As we know $U_\lambda$ as a representation of $O_N $ is isomorphic to either $R_{\kappa(\lambda)}$ or $R_{\kappa(\lambda)^{\dagger}}$.  Since $s(\kappa(\lambda)) = \kappa(\lambda) $, we deduce from the fact \ref{f_5.2} that $U_\lambda$ is isomorphic to $R_{\kappa(\lambda)}$. As $\lambda_n > 0$, we get that $r_n > l_n$ and $$\kappa(\lambda) = (l_n, l_{n-1}, \ldots l_1)^{t}.$$On the other hand, since $n$ is even $$\mathcal{F}(\lambda) = (l_n, l_{n-1}, \ldots l_1)^{t}.$$ So, $U_\lambda \cong R_{\mathcal{F}(\lambda)}$. 
Consider now the case $\lambda_n < 0$. Then $$g_d \rhd \xi_\lambda = (-1)^{n-1} \xi_\lambda = -\xi_\lambda,$$ which means that $g_d$ acts by $-1$ on $U_\lambda$, so $U_\lambda \cong R_{\kappa(\lambda)^{\dagger}}$. As $\lambda_n < 0$, we get that $l_n > r_n$ and $$\kappa(\lambda) = (r_n, l_{n-1}, \ldots l_1)^t,$$ so that $$\kappa(\lambda)^{\dagger} =  (l_n, l_{n-1}, \ldots l_1)^{t} =\mathcal{F}(\lambda).$$Again we get that $U_{\lambda} \cong R_{\mathcal{F}(\lambda)}$. 
Now consider the case when $n$ is odd. 
Again, first suppose that $\lambda_n > 0$. Then $$g_d \rhd \xi_\lambda = (-1)^{n} \xi_\lambda = - \xi_\lambda,$$which means that $g_d$ acts by $-1$ on $U_\lambda$, so $U_\lambda \cong R_{\kappa(\lambda)^{\dagger}}$. As $\lambda_n > 0$, we get that $$\kappa(\lambda)^{\dagger} = (r_n, l_{n-1}, \ldots, l_1)^{t}.$$On the other hand since $n$ is odd, $$\mathcal{F}(\lambda) = (r_n, l_{n-1}, \ldots l_1)^{t}.$$So, $U_\lambda \cong R_{\mathcal{F}(\lambda)}$. 
Finally, if $\lambda_n < 0$, then $$g_d \rhd \xi_\lambda = (-1)^{n-1} \xi_\lambda =  \xi_\lambda,$$which means that $g_d$ is the identity map on $U_\lambda$, so $U_\lambda \cong R_{\kappa(\lambda)}$.  As $\lambda_n < 0$, we get that  $$\kappa(\lambda) = (r_n, l_{n-1}, \ldots l_1)^t = \mathcal{F}(\lambda).$$So, $U_\lambda \cong R_{\mathcal{F}(\lambda)}$. 

We have just proved the theorem in case $N$ is even. We are left with the case where $N$ is odd. 
Due to the fact \ref{f_5.1}, we need to compute the action of $-\text{Id} \in O_N$ on $U_\lambda$. It is enough to study the action of $-\text{Id}$ on the vector $\xi_\lambda$. It acts on $\xi_\lambda$ by multiplying each factor of $\xi_\lambda$ by $-1$. So, $$-\text{Id} \rhd \xi_\lambda = (-1)^{\#\{\text{factors in } \hspace{1mm} \xi_\lambda\}}.$$Notice that the amount of factors in $\xi_\lambda$ equals the number of cells in the right part of the cell diagram $D_\lambda^{N}$. 

Now consider the case when $n$ is even. 
Then the total number of cells in the regular cell diagram $\lambda$ is even (it is equal to $n \cdot N$). Therefore, the parity of the number of cells in the right part of the diagram $D_\lambda^{N}$ equals the parity of the number of cells in the left part of the diagram. Then $$-\text{Id} \rhd \xi_\lambda = (-1)^{l_1 + l_2 + \ldots l_n} \xi_\lambda = (-1)^{|(l_n, l_{n-1}, \ldots, l_{1})^{t} |} \xi_\lambda = (-1)^{|\mathcal{F}(\lambda)|} \xi_\lambda.$$ Note that $r_n$ and $l_n$ have different parity (since $N$ is odd). So, we get $U_\lambda \cong R_{\mathcal{F}(\lambda)}$. 
Finally, if $n$ is odd, then the total number of cells in the regular cell diagram $D_\lambda^{N}$ is odd. Therefore, the parity of the number of cells in the right part of the diagram is the opposite of the parity of the number of cells in the left part of the diagram. Then, since $r_n$ and $l_n$ have different parity $$-\text{Id} \rhd \xi_\lambda = (-1)^{l_1 + l_2 + \ldots r_n} \xi_\lambda = (-1)^{|(r_n, l_{n-1}, \ldots, l_{1})^{t} |} \xi_\lambda = (-1)^{|\mathcal{F}(\lambda)|} \xi_\lambda.$$So,  $U_\lambda \cong R_{\mathcal{F}(\lambda)}$. 

Formally, we have not yet proved that the map $\mathcal{F}$ is a bijection. We see that $\mathcal{F}$ is injective because $U_\lambda$ are non-isomorphic for different $\lambda \in \Delta^N $ (see corollary \ref{cor_4.3}). 
It is also easy to see that $\mathcal{F}$ is surjective.  Let  $\nu \in SYD(N, n)$ be the short Young diagram of height $N$ that has at most $n$ columns. In case $n$ is even, set $$\lambda_i = \frac{N}{2}  - \nu^{t}_{n-i + 1}.$$
It is a direct check that $\lambda = (\lambda_i)_{1 \leq i \leq n} \in \Delta^N $  and $\mathcal{F}(\lambda) = \nu.$ 
In case $n$ is odd, set $$\lambda_n =  \nu_1^{t} - \frac{N}{2}$$ and for $1 \leq i < n$, $$\lambda_i = \frac{N}{2}  - \nu^{t}_{n-i + 1}.$$Then again one can directly check that $\lambda = (\lambda_i)_{1 \leq i \leq n} \in \Delta^N  $ and $\mathcal{F}(\lambda) = \nu$.

\end{proof}

\section{Regular cell tables  and Gelfand-Tsetlin patterns}
\label{sec_6}
\subsection{Bases in multiplicity spaces $U_\lambda$.}
There are two different bases we are going to construct in multiplicity space $U_\lambda$ for each $\lambda \in \Delta^N $. Recall that each multiplicity space $U_\lambda$ as a representation of Lie algebra $\mathfrak{o}_{N}(\mathbb{C})$ is either irreducible or is isomorphic to the sum of two non-isomorphic irreducible representations. 
Let $L_\beta$ be the irreducible representation of $\mathfrak{o}_{N}(\mathbb{C})$ given by the highest weight $\beta \in P_+$. One can naturally embed $\mathfrak{o}_{N-1}(\mathbb{C}) \subset \mathfrak{o}_N(\mathbb{C})$. Therefore, we can view $L_\beta$ as a representation of $\mathfrak{o}_{N-1}(\mathbb{C})$. 
It is known from the formulas $(25.34)$ and $(25.35)$ in \cite{Fulton}  that $L_\beta$ as a representation of $\mathfrak{o}_{N-1}(\mathbb{C})$  decomposes into the direct sum of irreducible subrepresentations without multiplicities: \begin{equation} \label{eq_6.1}
Res_{\mathfrak{o}_{N-1}}^{\mathfrak{o}_{N}} L_\beta = \bigoplus_{\mu \in P_+, \hspace{1mm} \mu \preccurlyeq \beta} L_{\mu},
\end{equation}
where  $$\mu \preccurlyeq \beta \Leftrightarrow \begin{cases} \beta_1 \geqslant \mu_1 \geqslant \beta_2 \geqslant \mu_2 \geqslant \ldots \geqslant \mu_{d-1} \geqslant \beta_{d} \geqslant | \mu_d|, \text{if} \hspace{1mm} N = 2d+1, \\  \beta_1 \geqslant \mu_1 \geqslant \beta_2 \geqslant \mu_2 \geqslant \ldots \geqslant \mu_{d-1} \geqslant |\beta_{d}|, \text{if} \hspace{1mm} N = 2d. \end{cases}$$ The essential part here is that each irreducible representation of $\mathfrak{o}_{N-1}$ in the decomposition of $L_\beta$ has a multiplicity not greater than $1$. Consider a nested family of orthogonal Lie algebras 
\begin{gather}
\label{eq_6.2}
\mathfrak{o}_3(\mathbb{C}) \subset \mathfrak{o}_4(\mathbb{C}) \subset \ldots \subset \mathfrak{o}_{N-1}(\mathbb{C}) \subset \mathfrak{o}_{N}(\mathbb{C}).
\end{gather}
From the formula \ref{eq_6.1} we can deduce the following corollary.
\begin{corollary}
   \label{cor_6.1}
Let $L_{\beta^{(N)}}$ be the highest weight representation of $\mathfrak{o}_{N}(\mathbb{C})$ corresponding to $\beta^{(N)} \in P_+$. Then it can be decomposed into the following direct sum of one-dimensional subspaces: 
\begin{equation}
    L_{\beta^{(N)}} = \bigoplus_{\underset{\beta^{(N)} \succcurlyeq \beta^{(N-1)} \succcurlyeq \ldots \succcurlyeq \beta^{(3)} \geqslant |z|}{(\beta^{(N)}, \ldots \beta^{(3)}, z), }} \mathbb{C}t_{(\beta^{(N)}, \ldots \beta^{(3)}, z)},
\end{equation}
where $\beta^{(i)}$ is the dominant weight of  $\mathfrak{o}_i(\mathbb{C})$ and $z \in \mathbb{Z}$. 
Vector $t_{(\beta^{(N)}, \ldots \beta^{(3)}, z)}$ is a nonzero element of the space  $L_{\beta^{(N)}}$, such that if we consider the decomposition of  $L_{\beta^{(N)}}$ into the direct sum of irreducible $\mathfrak{o}_i(\mathbb{C})$-modules, obtained by consecutive restrictions of the representation to the smaller Lie subalgebras from the nested family \ref{eq_6.2}, then it lies inside a summand isomorphic to $L_{\beta^{(i)}}$ and it is a weight vector under the action of $\mathfrak{o}_3(\mathbb{C})$ with the weight equal to $z$. 

The strings $(\beta^{(N)}, \ldots \beta^{(3)}, z)$ are called \textbf{Gelfand-Tsetlin patterns} and the basis of $L_{\beta^{(N)}}$ indexed by them is called \textbf{Gelfand-Tsetlin basis} with respect to the nested family of orthogonal Lie algebras \ref{eq_6.2} 
\end{corollary}
We introduce the following important notation. 

\begin{notation}
  Let $M$ be the finite-dimensional $\mathfrak{o}_N(\mathbb{C})$-module isomorphic to the direct sum of simple $\mathfrak{o}_N(\mathbb{C})$-modules without multiplicities . We denote by $\mathfrak{GT}(M)$ the Gelfand-Tsetlin basis of $M$ with respect to the nested family of orthogonal Lie algebras \ref{eq_6.2}.
\end{notation}
\begin{notation}
    Since each multiplicity space $U_\lambda$, as a $\mathfrak{o}_N(\mathbb{C})$-module, is either irreducible or is the sum of two non-isomorphic irreducibles, then it has the basis $\mathfrak{GT}(U_\lambda)$ indexed by Gelfand-Tsetlin patterns. We denote this set of Gelfand-Tsetlin patterns by $\mathfrak{GTP}(\nu)$, where $\nu$ is a short Young diagram, such that $U_\lambda \underset{O_N}{\cong} R_\nu$. In other words$\nu = \mathcal{F}(\mathcal{K}_N(\lambda))$.

\end{notation}

We will now denote the spinor representation with the base space $E$ by $S$ (as a vector space it is isomorphic to $\Lambda E$). 

There is another natural basis in multiplicity space $U_\lambda$. As mentioned above, it is natural to treat $U_\lambda$ as a subspace of singular vectors of weight $\lambda$ in $S^{\otimes N}$.  We will need the following notation. 

\begin{notation}
We denote the set of tuples $(\mu_1, \mu_2, \ldots, \mu_N) \in T^N$ (see \ref{cor_2.1}), such that $\sum_{i=1}^{N} \mu_i = \lambda$ by $T_\lambda^N$.
\end{notation}

Consider the irreducible subrepresentation $L_{\mu_1 + \mu_2 + \ldots + \mu_N}$ of $S^{\otimes N}$, such that  $(\mu_1, \mu_2, \ldots, \mu_N) \in T^N_{\lambda}$. The subrepresentation $L_{\mu_1 + \mu_2 + \ldots + \mu_N}$ is uniquely defined as the irreducible subrepresentation of the highest weight $\lambda$ in $L_{\mu_1 + \mu_2 + \ldots + \mu_{N-1}} \otimes S$. Similarly, $L_{\mu_1 + \mu_2 + \ldots + \mu_{N-1}}\subset S^{\otimes N-1}$  is uniquely defined as the irreducible  subrepresentation of the highest weight $\lambda - \mu_N$ in $L_{\mu_1 + \mu_2 + \ldots + \mu_{N-2}} \otimes S$ , etc. So, we see that the subspace $L_{\mu_1 + \mu_2 + \ldots + \mu_N} \subset S^{\otimes N}$ is  well-defined.  Taking the highest weight vectors inside irreducible subrepresentations $L_{\mu_1 + \mu_2 + \ldots + \mu_N} \subset S^{\otimes N}$ for all tuples $(\mu_1, \mu_2, \ldots, \mu_N) \in T^N_{\lambda}$ we obtain a basis in $U_\lambda$. This basis will be called \textbf{\textit{the principal basis}}. The principal basis is indexed by the set $T^N_{\lambda}$. It turns out that the set $T_\lambda^{N}$ has a nice combinatorial interpretation, which is well connected with the combinatorial interpretation of $\Delta^N $ discussed in paragraph 3, see \ref{prop_3.2}. 

\begin{notation}
    We denote the principal basis of the multiplicity space $U_\lambda$ by $\mathcal{PR}(U_\lambda)$.
\end{notation}

\subsection{Regular cell tables} 
\label{subs_cell_tables}
The set $T_\lambda^{N}$ indexes the principal basis in the multiplicity space $U_\lambda$. Similarly to how we identified the set $\Delta^N $ of the highest weights of irreducible subrepresentations of $S^{\otimes N}$ with the set $\mathfrak{D}(N, n)$ of the regular cell diagrams of length $N$ and height $n$, we will identify the set $T_\lambda^N$ for each $\lambda \in \Delta^N $ with the set of\textbf{ \textit{regular cell tables }} of shape $D_\lambda^N$. 

We will start with the necessary definitions. 

\begin{definition}
\label{def_6.1}
Let $D_1 = D(\textbf{l}^{(1)}, \textbf{r}^{(1)}) \in \mathfrak{D}(N_1, n)$ and $D_2 = D(\textbf{l}^{(2)}, \textbf{r}^{(2)}) \in \mathfrak{D}(N_2, n)$ be two regular cell diagrams of height $n$. We say that the diagram $D_1$ contains diagram $D_2$ (notation: $D_1 \supset D_2$)  iff  $l_i^{(1)}  \geqslant l_i^{(2)}$ and $r_i^{(1)}  \geqslant r_i^{(2)}$ for all $i \in \{1, 2, \ldots, n\}$. 
\end{definition}

\begin{definition}
\label{def_6.2}
\textbf{Regular cell table} of the shape $D^{(N)} \in \mathfrak{D}(N, n)$ is a sequence of regular cell diagrams $(D^{(1)}, D^{(2)}, \ldots, D^{(N)})$, such that $D^{(i)} \in \mathfrak{D}(i, n)$ and $D^{(i)} \supset D^{(j)}$ for $i \geqslant j$.  
We denote the set of regular cell tables of the shape $D^{(N)}$ by $\mathfrak{Ctab}(D^{(N)})$. 
\end{definition} 
The following proposition establishes a natural bijection between the sets $T_\lambda^N$ and $\mathfrak{Ctab}(D_\lambda^{N})$ for every $\lambda \in \Delta^N $.

\begin{proposition}
\label{prop_6.1}
For $\lambda \in \Delta^N $ consider a map $$\mathcal{I}_{\lambda}: T_\lambda^{N} \rightarrow \mathfrak{Ctab}(D_{\lambda}^{N})$$defined by the formula below.
\begin{gather} 
\label{eq_6.3}\mathcal{I}_{\lambda}:  (\mu_1, \mu_2, \ldots \mu_N) \mapsto (\mathcal{K}_1(\mu_1), \mathcal{K}_2(\mu_1 + \mu_2), \ldots, \mathcal{K}_i(\sum_{j=1}^{i} \mu_j), \ldots, \mathcal{K}_N (\lambda)),\end{gather}
where $(\mu_1, \mu_2, \ldots \mu_N) \in T_\lambda^N$.
Then the map $\mathcal{I}_\lambda$ is a bijection. 
 \end{proposition}
\begin{proof}
It is clear from the definition of the set $T_\lambda^N$ that for every $i$,  $$\lambda^{(i)}:= \sum_{j=1}^{i} \mu_j \in \Delta^{i} .$$
It is easy to see for each $i < N$, $$\mathcal{K}_i(\lambda^{(i)}) \subset \mathcal{K}_{i+1}(\lambda^{(i+1)})$$ since $$\lambda^{(i+1)}_k = \lambda_k^{(i)} \pm \frac{1}{2}$$for every $k \in \{1, 2, \ldots n\}$.  So either $$\begin{cases}r_k(\lambda^{(i+1)}, i+1) = r_k(\lambda^{(i)}, i) + 1 \\ l_k(\lambda^{(i+1)}, i+1) = l_k(\lambda^{(i)}, i) \end{cases}$$ or $$\begin{cases}r_k(\lambda^{(i+1)}, i+1) = r_k(\lambda^{(i)}, i)  \\ l_k(\lambda^{(i+1)}, i+1) = l_k(\lambda^{(i)}, i) +1 \end{cases}$$
So, the map $\mathcal{I}_{\lambda}$ is well-defined and is injective, because all the maps $\mathcal{K}_i$ are bijective by \ref{prop_3.2}. 
Let us prove the surjectivity of $\mathcal{I}_\lambda$. Consider $x = (D^{(1)}, D^{(2)}, \ldots, D^{(N)}) \in \mathfrak{Ctab}(D_{\lambda}^{N})$. Let $D^{(i)} = D(\textbf{l}^{(i)}, \textbf{r}^{(i)})$. Set $$(\mu_i)_j = \frac{1}{2}( r^{(i)}_j - r^{(i-1)}_j + l^{(i-1)}_j - l^{(i)}_j).$$
Then, one can directly check that $(\mu_1, \mu_2, \ldots \mu_N) \in T_\lambda^N$ and $$\mathcal{I}_{\lambda}((\mu_1, \mu_2, \ldots \mu_N)) = x.$$Hence, $\mathcal{I}_\lambda$ is bijective.
\end{proof}
\begin{example}\label{ex_6.1}
For clarity, let us visualize the notion of a regular cell table of some fixed shape. 
Consider the regular cell table $x = (D^{(1)}, D^{(2)}, \ldots D^{(7)})$ of the shape $D(\textbf{l}, \textbf{r}) \in \mathfrak{D}(7,4)$, where $\textbf{r} = (r_1, r_2, r_3, r_4) = (5, 4, 4, 3)$ and $\textbf{l} = (l_1, l_2, l_3, l_4) = (2, 3, 3, 4)$, defined in the following way. 
Let $\mu_1 = (\frac{1}{2}, \frac{1}{2},\frac{1}{2}, -\frac{1}{2})$, $\mu_2 = (\frac{1}{2}, \frac{1}{2},\frac{1}{2}, \frac{1}{2})$, $\mu_3 = (\frac{1}{2}, -\frac{1}{2},-\frac{1}{2}, -\frac{1}{2})$, $\mu_4 = (\frac{1}{2}, \frac{1}{2},\frac{1}{2}, \frac{1}{2})$, $\mu_5 = (\frac{1}{2}, -\frac{1}{2},-\frac{1}{2}, -\frac{1}{2})$, $\mu_6 = (-\frac{1}{2}, \frac{1}{2},\frac{1}{2}, -\frac{1}{2})$, $\mu_7 = (-\frac{1}{2}, -\frac{1}{2},-\frac{1}{2}, \frac{1}{2})$. Then $$x = \mathcal{I}_\lambda((\mu_1, \mu_2, \ldots, \mu_7)) \in \mathfrak{Ctab}(D(\textbf{l}, \textbf{r})),$$ where $\lambda = (\frac{3}{2}, \frac{1}{2}, \frac{1}{2}, -\frac{1}{2})$.
We can visualize $x$ the following way:

\centering

\begin{tikzpicture}
\draw[thick,<->] (6,-0.5) -- (6, 4.5) node[anchor=north west] {};
\draw (2,0) rectangle (3,1);
\draw (2.75, 0.5) circle (0pt)  node[anchor=east]{$6$};
\draw (3,0) rectangle (4,1);
\draw (3.75, 0.5) circle (0pt)  node[anchor=east]{$5$};
\draw (4,0) rectangle (5,1);
\draw (4.75, 0.5) circle (0pt)  node[anchor=east]{$3$};
\draw (5,0) rectangle (6,1);
\draw (5.75, 0.5) circle (0pt)  node[anchor=east]{$1$};
\draw (6,0) rectangle (7,1);
\draw (6.75, 0.5) circle (0pt)  node[anchor=east]{$2$};
\draw (7,0) rectangle (8,1);
\draw (7.75, 0.5) circle (0pt)  node[anchor=east]{$4$};
\draw (8,0) rectangle (9,1);
\draw (8.75, 0.5) circle (0pt)  node[anchor=east]{$7$};

\draw (3,1) rectangle (4,2);
\draw (3.75, 1.5) circle (0pt)  node[anchor=east]{$7$};
\draw (4,1) rectangle (5,2);
\draw (4.75, 1.5) circle (0pt)  node[anchor=east]{$5$};
\draw (5,1) rectangle (6,2);
\draw (5.75, 1.5) circle (0pt)  node[anchor=east]{$3$};
\draw (6,1) rectangle (7,2);
\draw (6.75, 1.5) circle (0pt)  node[anchor=east]{$1$};
\draw (7,1) rectangle (8,2);
\draw (7.75, 1.5) circle (0pt)  node[anchor=east]{$2$};
\draw (8,1) rectangle (9,2);
\draw (8.75, 1.5) circle (0pt)  node[anchor=east]{$4$};
\draw (9,1) rectangle (10,2);
\draw (9.75, 1.5) circle (0pt)  node[anchor=east]{$6$};

\draw (3,2) rectangle (4,3);
\draw (3.75, 2.5) circle (0pt)  node[anchor=east]{$7$};
\draw (4,2) rectangle (5,3);
\draw (4.75, 2.5) circle (0pt)  node[anchor=east]{$5$};
\draw (5,2) rectangle (6,3);
\draw (5.75, 2.5) circle (0pt)  node[anchor=east]{$3$};
\draw (6,2) rectangle (7,3);
\draw (6.75, 2.5) circle (0pt)  node[anchor=east]{$1$};
\draw (7,2) rectangle (8,3);
\draw (7.75, 2.5) circle (0pt)  node[anchor=east]{$2$};
\draw (8,2) rectangle (9,3);
\draw (8.75, 2.5) circle (0pt)  node[anchor=east]{$4$};
\draw (9,2) rectangle (10,3);
\draw (9.75, 2.5) circle (0pt)  node[anchor=east]{$6$};

\draw (5,3) rectangle (4,4);
\draw (4.75, 3.5) circle (0pt)  node[anchor=east]{$7$};
\draw (6,3) rectangle (5,4);
\draw (5.75, 3.5) circle (0pt)  node[anchor=east]{$6$};
\draw (7,3) rectangle (6,4);
\draw (6.75, 3.5) circle (0pt)  node[anchor=east]{$1$};
\draw (8,3) rectangle (7,4);
\draw (7.75, 3.5) circle (0pt)  node[anchor=east]{$2$};
\draw (9,3) rectangle (8,4);
\draw (8.75, 3.5) circle (0pt)  node[anchor=east]{$3$};
\draw (10,3) rectangle (9,4);
\draw (9.75, 3.5) circle (0pt)  node[anchor=east]{$4$};
\draw (11,3) rectangle (10,4);
\draw (10.75, 3.5) circle (0pt)  node[anchor=east]{$5$};
\end{tikzpicture}

 Here $D^{(i)} \in \mathfrak{D}(i, n)$ is the regular cell diagram formed by all the cells with the numbers less or equal to $i$. 
\end{example}
\subsection{Semi-standard short Young tables and regular cell tables.}

Let $V$ be the $N$-dimensional vector space. Consider the orthogonal group $O_N$ acting on this space. Let $$\textbf{v} = (v_1, v_2, \ldots v_N)$$ be an orthonormal basis of $V$. Consider the embedding of orthogonal groups $O_{N-1} \hookrightarrow O_N$, such that the image of $O_{N-1}$ is a subgroup of the operators in $O_N$ that acts trivially on the last vector $v_N$ from the orthonormal basis $\textbf{v}$.
\begin{proposition}
\label{prop_6.2}
Multiplicity space $U_{\lambda}$ where $\lambda \in \Delta^N $ as a representation of the group $O_{N-1}$ admits the following decomposition into the direct sum of irreducible representations
\begin{gather}
    U_\lambda \underset{O_{N-1}}{\cong} \bigoplus_{\underset{\lambda - \beta \in P[S]}{\beta \in \Delta^{N-1} ,} } U_\beta
\end{gather}
\end{proposition}
\begin{proof}
From the decomposition formulas \ref{eq_2.2} and \ref{eq_2.4} we obtain 
\begin{gather}
\label{eq_6.4}
S^{\otimes N} =  S^{\otimes (N-1)} \otimes S \underset{\mathfrak{o}(\tilde{E})}{\cong}   \bigoplus_{\beta \in \Delta^{N-1} } U_\beta \otimes (L_\beta \otimes S) \underset{\mathfrak{o}(\tilde{E})}{\cong}\\\underset{\mathfrak{o}(\tilde{E})}{\cong}  \bigoplus_{\beta \in \Delta^{N-1} } U_\beta \otimes (\bigoplus_{\underset{\beta + \mu \in P_+ }{\mu \in P[S],}} L_{\beta + \mu})
\end{gather}
On the other hand, \begin{gather}
\label{eq_6.5}
S^{\otimes N} \underset{\mathfrak{o}(\tilde{E})}{\cong} \bigoplus_{\lambda \in \Delta^{N} } U_\lambda \otimes L_\lambda.\end{gather}
So, it follows that, as the vector spaces\begin{gather}
    U_\lambda \cong \bigoplus_{\underset{\lambda - \beta \in P[S]}{\beta \in \Delta^{N-1} ,} } U_\beta
\end{gather}
We want to show that these spaces are also isomorphic as $O_{N-1}$ modules. 
It is clear that the group $O_{N-1}$ acts only on the first $N-1$ factors of $S^{\otimes N}$, so the isomorphisms in \ref{eq_6.4}, \ref{eq_6.5} are also $O_{N-1}$-module isomorphisms. Therefore, $$\bigoplus_{\lambda \in \Delta^{N} } U_\lambda \otimes L_\lambda \underset{O_{N-1} \times \mathfrak{o}(\tilde{E})}{\cong} S^{\otimes N} \underset{O_{N-1} \times \mathfrak{o}(\tilde{E})}{\cong} \bigoplus_{\beta \in \Delta^{N-1} } U_\beta \otimes (\bigoplus_{\underset{\beta + \mu \in P_+ }{\mu \in P[S],}} L_{\beta + \mu}).$$
Hence, for each $\lambda \in \Delta^N $,

$$U_\lambda \otimes L_\lambda \underset{O_{N-1} \times \mathfrak{o}(\tilde{E})}{\cong}  \bigoplus_{\underset{\lambda - \beta \in P[S]}{\beta \in \Delta^{N-1} ,} } U_\beta \otimes L_\lambda.$$So, we conclude that \begin{gather}
    U_\lambda \underset{O_{N-1}} {\cong} \bigoplus_{\underset{\lambda - \beta \in P[S]}{\beta \in \Delta^{N-1} ,} } U_\beta
\end{gather}
and we are done. 
\end{proof}

We can go further by decomposing each $U_\beta$ into the sum of irreducible representations of $O_{N-2} \hookrightarrow O_{N-1}$ and so on. We may consider the nested family of orthogonal groups:
\begin{gather}
\label{eq_6.6}
    O_1 \subset O_2 \subset \ldots \subset O_{N-1} \subset O_{N}
\end{gather}
In the end, we obtain the decomposition of $U_\lambda$ into the direct sum of irreducible representations of the group $O(1) \cong \mathbb{Z}/{2\mathbb{Z}}$. But all irreducible representations of this group are one-dimensional. Hence, the following proposition. 
\begin{corollary}
\label{cor_6.2}
Let $U_{\lambda^{(N)}}$ be the multiplicity space corresponding to $\lambda^{(N)} \in \Delta^{N} $. Then it can be decomposed into the following direct sum of one-dimensional subspaces: 
\begin{equation}
    U_{\lambda^{(N)}} = \bigoplus_{(\lambda^{(1)}, \lambda^{(2)}, \ldots, \lambda^{(N)})} \mathbb{C}v_{(\lambda^{(1)}, \lambda^{(2)}, \ldots, \lambda^{(N)})},
    \end{equation}
where $\lambda^{(i)} \in \Delta^{i} $ and $\lambda^{(i+1)} - \lambda^{(i)} \in P[S]$. 
Vector $v_{(\lambda^{(1)}, \lambda^{(2)}, \ldots, \lambda^{(N)})} \neq 0$ is an element in the space  $U_{\lambda^{(N)}}$ such that, if we consider the decomposition of  $U_{\lambda^{(N)}}$ into the direct sum of irreducible $O_{i}$-modules, obtained by consecutive restrictions of the representation to the smaller orthogonal groups from the nested family \ref{eq_6.6}, then it lies inside a summand isomorphic to $U_{\lambda^{(i)}}$. 

The basis of $U_{\lambda^{(N)}} $ formed by the vectors $v_{(\lambda^{(1)}, \lambda^{(2)}, \ldots, \lambda^{(N)})}$ is called \textbf{Gelfand-Tsetlin basis} with respect to the nested family of orthogonal groups \ref{eq_6.6}, denoted by $GT(U_{\lambda^{(N)}})$.
Tuples $(\lambda^{(1)}, \lambda^{(2)}, \ldots, \lambda^{(N)})$ are in one-to-one correspondence with the regular cell tables  of the shape $D_{\lambda^{(N)}}^{N}$. The correspondence is given by the formula  
$$(\lambda^{(1)}, \lambda^{(2)}, \ldots, \lambda^{(N)}) \mapsto (\mathcal{K}_1(\lambda^{(1)}), \mathcal{K}_2(\lambda^{(2)}), \ldots, \mathcal{K}_N(\lambda^{(N)})).$$
\end{corollary}

\begin{proposition}
\label{prop_6.3}    
Up to rescaling, $GT(U_\lambda) = \mathcal{PR}(U_\lambda)$.
\end{proposition}

So, as we see from the corollary \ref{cor_6.2} $GT(U_\lambda)$, where $\lambda \in \Delta^N $, is indexed by the set of the regular cell tables of the shape $D_\lambda^N$ - $\mathfrak{Ctab}(D_\lambda^N)$. Another way to explain it is to say that since in the proposition \ref{prop_6.1} we established the bijection between the set $T_\lambda^N$ and $\mathfrak{Ctab}(D_\lambda^N)$, then $\mathcal{PR}(U_\lambda)$ is indexed by the set $\mathfrak{Ctab}(D_\lambda^N)$ and due to the proposition \ref{prop_6.3} so is $GT(U_\lambda)$. 
But $U_\lambda$ is an irreducible representation of the group $O_N$. So, $GT(U_\lambda)$ has its own natural indexing set called \textbf{\textit{semi-standard short Young tables.}}
Let us discuss this set in detail. 

Let $\nu \in SYD(N, n)$ be the short Young diagram of height $N$ and length $n$. Consider an irreducible representation $R_{\nu}$ of the group $O_N$, corresponding to that diagram. We can realize this representation as a multiplicity space $U_\lambda$, where $$\lambda  = \mathcal{K}_{N}^{-1}(\mathcal{F}^{-1}(\nu))\in \mathfrak{} \Delta^N .$$ Due to proposition \ref{prop_6.2}, we know the decomposition of $U_\lambda$ into the sum of irreducible representations of the group $O_{N-1}$. Hence, we get the following corollary.  \begin{corollary}
\label{cor_6.3}
     Let $\nu$ be the short Young diagram of height $N$ and length $n$. Then irreducible representation $R_{\nu}$ of the group $O_N$ corresponding to the diagram $\nu$ admits the following decomposition into the sum of irreducible representations of the group $O_{N-1}$: 
     \begin{gather}
     \label{eq_6.7}
         R_\nu \cong \bigoplus_{\underset{\nu \supset \rho, \nu - \rho = \text{horizontal strip}}{\rho \in SYD(N-1, n)}} R_\rho
     \end{gather}
 \end{corollary}
\begin{proof}
Let  $\lambda  = \mathcal{K}_{N}^{-1}(\mathcal{F}^{-1}(\nu))\in \mathfrak{} \Delta^N $. Then from proposition\ref{prop_6.2} and theorem \ref{th_5.1} we get the following chain of isomorphisms of the $O_{N-1}$- modules. 
   \begin{gather} 
   R_\nu \underset{O_N}{\cong}  U_\lambda \underset{O_{N-1}}{\cong} \bigoplus_{\underset{\lambda - \mu \in P[S]}{\mu \in \Delta^{N-1} }} U_\mu \underset{O_{N-1}}{\cong} \bigoplus_{\underset{\hspace{1mm} \lambda - \mu \in P[S]}{\mu \in \Delta^{N-1} \hspace{1mm}}} R_{\mathcal{F}(\mathcal{K}_{N-1}(\mu))} \underset{O_{N-1}}{\cong}\bigoplus_{\underset{\nu \supset \rho, \nu - \rho = \text{horizontal strip}}{\rho \in SYD(N-1, n)}} R_\rho.
   \end{gather}

\end{proof}

We need the following definition. 

\begin{definition}
   \label{def_6.3}
    \textbf{Semi-standard short Young table} of the shape $\nu \in SYD(N, n)$ is the sequence of the short Young diagrams $(\nu^{(1)}, \nu^{(2)}, \ldots, \nu^{(N)})$, such that $\nu^{(i)} \in SYD(i, n), \nu^{(N)} = \nu$, $\nu^{(i)} \subset \nu^{(i+1)}$ and $\nu^{(i+1)} - \nu^{(i)}$ is a horizontal strip for all $1 \leqslant i < N$. 
    
We denote the set of semi-standard short Young tables of shape $\nu$ by $SSSYT(\nu)$. 

Note that here it is very important to consider how we treat $\nu$. If $\nu \in SYD(N, n)$, then it belongs to all the sets $SYD(M, n)$ where $M > N$. But in our case, $\nu$ is always an index of some irreducible representation of the orthogonal group, so we always treat $\nu$ as an element of $SYD(N, n)$, where $N$ is the dimension of the group. 
\end{definition}

\begin{corollary}
 \label{cor_6.4}
 Let $R_\nu$ be the irreducible representation of the group $O_N$, corresponding to the short Young diagram $\nu$. Then $GT(R_\nu)$ is naturally indexed by the set $SSSYT(\nu)$.

Take $\lambda = \mathcal{K}_N^{-1}(\mathcal{F}^{-1}(\nu)) \in \Delta^N $. Then $U_\lambda \underset{O_N}{\cong} R_\nu$. As we saw in \ref{cor_6.2} $GT(U_\lambda)$ is indexed by the set $\mathfrak{Ctab}(D_\lambda^N)$. By identifying $U_\lambda$ with $R_\nu$ we get the natural bijection $$\mathcal{Y}_\lambda: \mathfrak{Ctab}(D_\lambda^N) \rightarrow SSSYT(\nu)$$ determined by the following formula 
\begin{gather} \label{eq_6.10} 
\mathcal{Y}_\lambda: (D^{(1)}, D^{(2)}, \ldots, D^{(N)} = D^N_\lambda) \mapsto (\mathcal{F}(D^{(1)}), \mathcal{F}(D^{(2)}), \ldots, \mathcal{F}(D^{(N)}) = \nu)
\end{gather}
\end{corollary}
    
We established the bijection between regular cell tables and semistandard short Young diagrams. 

\subsection{Semi-standard short Young tables and Gelfand-Tsetlin patterns.}
\label{subs_bij}

We have seen in the corollary \ref{cor_6.1} that the Gelfand-Tsetlin patterns index the Gelfand-Tsetlin basis $\mathfrak{GT}(U_\lambda)$ in the multiplicity space $U_\lambda$. Unfortunately, bases $\mathfrak{GT}(U_\lambda)$ and $GT(U_\lambda)$ differ. Therefore, to obtain a bijection between Gelfand-Tsetlin patterns and regular cell tables, we need to establish some natural bijection between bases. 

Let us first consider an example that shows that bases $\mathfrak{GT}(U_\lambda)$ and $GT(U_\lambda)$ differ. 

\begin{example}
Let $\nu = (4, 1) \in SYD(4, 4)$ be the short Young diagram of height $4$ and length $4$. It corresponds to the highest weight $\lambda = (1,1,1,0) \in \Delta^{4}$ in the sense that $\mathcal{F}(\mathcal{K}_{4}(\lambda)) = \nu$.

According to the corollary \ref{cor_6.3} we have the following decomposition of the multiplicity space $U_\lambda \underset{O_4}{\cong}R_{\nu}$ into the direct sum of the irreducible $O_3$-modules: 
\begin{gather} 
\label{eq_6.8}
U_{\lambda} \underset{O_3}{=} R_{(4,1)} \oplus R_{(4,0)} \oplus R_{(3, 1)} \oplus R_{(3, 0)}\oplus R_{(2, 1)}  \oplus R_{(2, 0)} \oplus R_{(1, 1)} \oplus R_{(1, 0)} 
\end{gather}
On the other hand, $U_\lambda$ as $\mathfrak{o}_4$-module is isomorphic to the sum of two simple modules: $$U_\lambda \underset{\mathfrak{o}_4}{\cong} L_{(4,1)} \oplus L_{(4, -1)}$$
These two submodules are isomorphic as $\mathfrak{o}_3$-modules.  Each of them admits the following decomposition: 

$$L_{(4, -1)} \underset{{\mathfrak{o}_3}}{\cong} L_{(4,1)} \underset{{\mathfrak{o}_3}}{\cong} L_{4} \oplus L_{3} \oplus L_{2} \oplus L_{1}$$
So, $U_\lambda$ as $\mathfrak{o}_3$-module admits the following decomposition: 

$$U_\lambda \underset{{\mathfrak{o}_3}}{=} L_4^{(4,1)} \oplus L_4^{(4, -1)} \oplus L_3^{(4,1)} \oplus L_3^{(4, -1)} \oplus L_2^{(4,1)} \oplus L_2^{(4, -1)} \oplus L_1^{(4,1)} \oplus L_1^{(4, -1)}$$
Here the lower index of the summand is its highest weight and the upper index represents the highest weight of $\mathfrak{o}_4$-submodule that it came from. 
Each irreducible $O_3$ module from the decomposition \ref{eq_6.8} is an irreducible $\mathfrak{o}_3$ module. 
$$R_{(4,1)} \underset{{\mathfrak{o}_3}}{\cong} R_{(4,0)} \underset{{\mathfrak{o}_3}}{\cong} L_4$$
$$R_{(3,1)} \underset{{\mathfrak{o}_3}}{\cong} R_{(3,0)} \underset{{\mathfrak{o}_3}}{\cong} L_3$$
$$R_{(2,1)} \underset{{\mathfrak{o}_3}}{\cong} R_{(2,0)} \underset{{\mathfrak{o}_3}}{\cong} L_2$$
$$R_{(1,1)} \underset{{\mathfrak{o}_3}}{\cong} R_{(1,0)} \underset{{\mathfrak{o}_3}}{\cong} L_1$$
Submodules $L_4^{(4,1)}$ and $L_4^{(4, -1)}$  don't coincide with either $R_{(4,1)}$ or $R_{(4, 0)}$. But their sum $L_4^{(4,1)} \oplus L_4^{(4, -1)}$ coincides with $R_{(4,1)} \oplus R_{(4, 0)}$ as a subspace of $U_\lambda$. It follows from the fact that both sums form an isotypic component of $L_4$ in $U_\lambda$. 

To prove that these decompositions of an isotypic component of $L_4$ are different, we need to make the following observation. Take an element $ -\text{Id} = g \in O_3$. We know because of \ref{f_5.1} that it acts by $-1$ on the $O_3$ submodule $R_{(4,1)}$ and by identity on $R_{(4,0)}$. 

On the other hand, if we look at $g$ as an element of $O_4$, then we see that it does not belong to $SO_4$ and therefore neither $L_{(4,1)}$  nor $L_{(4, -1)}$ are invariant under its action. One can directly check that $g$ sends the highest vector of $L_{(4,1)}$ to the highest weight vector of $L_{(4, -1)}$. 

Observe that the highest vector of $L_4^{(4,1)}$ coincides with the highest weight vector of $L_{(4,1)}$ and, similarly, the highest vector of $L_4^{(4, -1)}$ coincides with the highest weight vector of $L_{(4,-1)}$. We conclude that $g$ permutes the submodules $L_4^{(4, 1)}$and  $L_4^{(4, -1)}$, since $g$ lies in the center of $O_3$. Therefore, $L_4^{(4,1)} \oplus L_4^{(4, -1)}$ and $R_{(4,1)} \oplus R_{(4, 0)}$ are different decompositions of the isotypic component of $L_4$ in $U_\lambda$. 
\end{example}

   This example shows the difference between bases $GT(U_\lambda)$ and $\mathfrak{GT}(U_\lambda)$. Establishing some random bijection between these bases is not interesting at all. We want to establish a bijection that will reflect the similarities between these bases. Let us consider another example. 

\begin{example}
Let $\nu = (4, 1, 1, 1,1) \in SYD(6, 4)$ be the short Young diagram of height $6$ and length $4$. It corresponds to the highest weight $\lambda = (2,2,2,-2) \in \Delta^{6}$ in the sense that $\mathcal{F}(\mathcal{K}_{6}(\lambda)) = \nu$.

 According to the corollary \ref{cor_6.3} we have the following decomposition of the multiplicity space $U_\lambda \underset{O_6}{\cong} R_{\nu}$ into the direct sum of the irreducible $O_5$ modules. 

$$U_\lambda \underset{O_5}{=} R_{(4,1,1,1)} \oplus R_{(3,1,1,1)}  \oplus R_{(2,1,1,1)} \oplus R_{(1,1,1,1)} \oplus R_{(1,1,1,1,1)}$$

As a $\mathfrak{o}_{6}$ module $U_\lambda \underset{{\mathfrak{o}_6}}{\cong} L_{(4,0, 0)}$, so it decomposes into the following sum of simple $\mathfrak{o}_5$ modules: 
$$U_\lambda \underset{{\mathfrak{o}_5}}{=} L_{(4, 0)} \oplus L_{(3,0)} \oplus L_{(2, 0)} \oplus L_{(1,0)} \oplus L_{(0,0)}$$
We also have the following isomorphisms of $\mathfrak{o}_5$-modules: 

$$R_{(4,1,1,1)} \underset{{\mathfrak{o}_5}}{\cong} L_{(4,0)}$$
$$R_{(3,1,1,1)} \underset{{\mathfrak{o}_5}}{\cong} L_{(3,0)}$$
$$R_{(2,1,1,1)} \underset{{\mathfrak{o}_5}}{\cong} L_{(2,0)}$$
$$R_{(1,1,1,1)} \underset{{\mathfrak{o}_5}}{\cong} L_{(1,0)}$$
$$R_{(1,1,1,1,1)} \underset{{\mathfrak{o}_5}}{\cong} L_{(0,0)}$$
We conclude that the decomposition of the multiplicity space $U_\lambda$ into the direct sum of simple $\mathfrak{o}_5$ modules coincides with the decomposition of $U_\lambda$ into the direct sum of simple $O_5$ modules. 
\end{example}
This example shows that sometimes decompositions into the direct sum of simple $\mathfrak{o}_N$ modules and the sum of simple $O_N$ modules agree with each other. 

Consider the multiplicity space $U_\lambda$, where $\lambda \in \Delta^N $. Set $$\nu = \mathcal{F}(\mathcal{K}_N(\lambda)) \in SYD(N, n)$$ For each $k \leqslant N$ we can decompose $U_\lambda$ into the direct sum of simple $O_k$ modules: 

$$U_\lambda = \bigoplus_{(\nu^{(k)}, \nu^{(k+1)}, \ldots, \nu^{(N)}= \nu)} M_{(\nu^{(k)}, \nu^{(k+1)}, \ldots, \nu^{(N)})},$$
where  $M_{(\nu^{(k)}, \nu^{(k+1)}, \ldots, \nu^{(N)})}$ is a $O_k$-submodule of $U_\lambda$ isomorphic to $R_{\nu^{(k)}}$ uniquely determined by the following property: the subspace $M_{(\nu^{(k)}, \nu^{(k+1)}, \ldots, \nu^{(N)})} \subset U_\lambda$ lies inside the $O_i$ submodule of $U_\lambda$ isomorphic to $R_{\nu^{(i)}}$ for all $i \in \{k, k+1, \ldots, N\}$. 

Similarly, for each $3 \leqslant k \leqslant N$ we can decompose $U_\lambda$ into the direct sum of simple $\mathfrak{o}_k$-modules: 
$$U_\lambda = \bigoplus_{(\beta^{(N)}, \beta^{(N-1)}, \ldots, \beta^{(k)})} H_{(\beta^{(N)}, \beta^{(N-1)}, \ldots, \beta^{(k)})},$$
where $H_{(\beta^{(N)}, \beta^{(N-1)}, \ldots, \beta^{(k)})}$ is a $\mathfrak{o}_k$-submodule of $U_\lambda$ isomorphic to $L_{\beta^{(k)}}$  uniquely determined by the following property: the subspace  $H_{(\beta^{(N)}, \beta^{(N-1)}, \ldots, \beta^{(k)})} \subset U_\lambda$ lies inside the $\mathfrak{o}_i$-submodule of $U_\lambda$ isomorphic to $L_{\beta^{(i)}}$ for all $i \in \{k, k+1, \ldots, N\}$. 
\begin{definition}
   \label{def_6.4}
    We call $O_k$-submodule of $U_\lambda$ \textbf{decent}, if it is a sum of some $O_k$-submodules   $M_{(\nu^{(k)}, \nu^{(k+1)}, \ldots, \nu^{(N)})}$ from the decomposition above. 
    
    Similarly, we call $\mathfrak{o}_k$ submodule of $U_\lambda$ \textbf{decent}, if it is the sum of some $\mathfrak{o}_k$ submodules $H_{(\beta^{(N)}, \beta^{(N-1)}, \ldots, \beta^{(k)})}$ from the above decomposition. 
\end{definition}

We want our bijection between $GT(U_\lambda)$ and $\mathfrak{GT}(U_\lambda)$ to have the following properties that we call \textbf{natural conditions}:
\begin{itemize}
    \item For any $k \in \{3, \ldots N\}$, if a decent $O_k$ submodule $M \subset U_\lambda$ is equal as a subspace to a decent $\mathfrak{o}_k$-submodule $H \subset U_\lambda$, then the bijection sends the basis of $M$, which is a subset of $GT(U_\lambda)$, to the basis of $H$, which is a subset of $\mathfrak{GT}(U_\lambda)$ .  
    \item For any $k \in \{3, \ldots N\}$, the bijection sends the basis of any decent $O_k$ submodule $M \subset U_\lambda$, which is a subset of $GT(U_\lambda)$,  to the basis of some decent $\mathfrak{o}_k$ submodule $H$, which is a subset of $\mathfrak{GT}(U_\lambda)$.
    \item For any $k \in \{3, \ldots N\}$, let $u$ be a vector vector from $GT(U_\lambda)$, which lies in a $O_k$ submodule of $U_\lambda$ isomorphic to $R_{\nu^{(k)}}$, where $\nu^{(k)} \in SYD(k, n)$ and $\nu^{(k)} = \nu^{(k)\dagger}$, then consider the action of an element $g = \begin{bmatrix}
        -\text{Id}_{k-1} & 0\\
        0 & 1
    \end{bmatrix} \in O_k$ on $u$. If $$g \rhd u = u,$$then the image of the vector $u$ under bijection lies inside a $\mathfrak{o}_k$ submodule of $U_\lambda$ isomorphic to $L_{s(\nu^{(k)})}$.  Otherwise, $$g \rhd u = -u,$$ and the image of the vector $u$ lies inside $\mathfrak{o}_k$-submodule of $U_\lambda$ isomorphic to $L_{s(\nu^{(k)})^{\sigma}}$
    \item Let $u$ be a vector from $GT(U_\lambda)$, which lies in a $O_2$ submodule of $U_\lambda$ isomorphic to $R_{\nu^{(2)}}$, where $\nu^{(2)} \in SYD(2, n)$. Then the image of $u$ under the bjection is a weight vector under the action of $\mathfrak{o}_3$ and the module of its weight is equal to $|s(\nu^{(2)})|$. Moreover, if $u$ lies in the trivial $O_1$submodule of $U_\lambda$, then the weight of its image equals $|s(\nu^{(2)})|$ and $-|s(\nu^{(2)})|$ otherwise. 
\end{itemize}

\begin{remark}
The last two conditions are not exactly "natural". In reality, each vector of $GT(U_\lambda)$, which lies in the $O_k$ submodule of $U_\lambda$ isomorphic to $R_{\nu^{(k)}}$, where $\nu^{(k)} \in SYD(k, n)$ and $\nu^{(k)} = \nu^{(k)\dagger}$, can be sent to the vector of $\mathfrak{GT}(U_\lambda)$, which lies in either $L_{\nu^{(k)}}$ or $L_{(\nu^{(k)})^{\sigma}}$ . The last two conditions essentially fix the way we choose between these two options at each step. In other words, the first two conditions define a family of the "truly natural" bijections between $GT(U_\lambda)$ and  $\mathfrak{GT}(U_\lambda)$. The last two conditions fix the choice of a certain bijection from this family.  It appears that no bijection from this family is better than the others.  
\end{remark}

\begin{proposition}
\label{prop_6.4}
Let $\lambda \in \Delta^N $  and set $\nu = \mathcal{F}(\mathcal{K}_{N}(\lambda)) \in SYD(N,n)$. There is a unique bijection $$\mathcal{J_\lambda}: GT(U_\lambda) \rightarrow \mathfrak{GT}(U_\lambda),$$satisfying the natural conditions above. It is determined by the following formula: 
    $$\mathcal{J}_\lambda: v_{(\nu^{(1)}, \nu^{(2)}, \ldots, \nu^{(N)})} \mapsto t_{(\beta^{(N)}, \beta^{(N-1)}, \ldots, \beta^{(3)}, z)},$$
where $$(\nu^{(1)}, \nu^{(2)}, \ldots, \nu^{(N)}) \in SSSYT(\nu)$$ and$$(\beta^{(N)}, \beta^{(N-1)}, \ldots, \beta^{(3)}, z) \in \mathfrak{GTP}(\nu)$$ satisfy the following equations: 

For $k \geqslant 3$ 
\begin{gather} 
\label{eq_6.9}
\beta^{(k)} = \begin{cases} 
    s(\nu^{(k)}), \text{if} \hspace{1mm} \nu^{(k)} \neq {\nu^{(k)}}^{\dagger}  \\
   \nu^{(k)}, \text{if} \hspace{1mm} \nu^{(k)} = {\nu^{(k)}}^{\dagger} \text{and} \hspace{1mm} (-1)^{|\nu^{(k-1)}|} = 1\\
    (\nu^{(k)})^{\sigma},  \text{if} \hspace{1mm} \nu^{(k)} = {\nu^{(k)}}^{\dagger} \text{and} \hspace{1mm} (-1)^{|\nu^{(k-1)}|} = -1
\end{cases}
\end{gather}
and
\begin{gather} \label{eq_6.10}
z = \begin{cases}
    |s(\nu^{(2)})|,\hspace{1mm} \text{if} \hspace{1mm} \nu^{(1)} = \emptyset \\
    -|s(\nu^{(2)})|, \hspace{1mm}\text{if} \hspace{1mm} \nu^{(1)} \neq \emptyset.
\end{cases}
\end{gather}

\end{proposition}
\begin{corollary}
\label{cor_6.5}
We obtain a natural bijection between sets $SSSYT(\nu)$ and  $\mathfrak{GTP}(\nu)$ 
$$\mathcal{\bar J}_{\nu} : SSSYT(\nu) \overset{\sim}{\longrightarrow}  \mathfrak{GTP}(\nu),$$$$(\nu^{(1)}, \nu^{(2)}, \ldots, \nu^{(N)}) \mapsto (\beta^{(N)}, \beta^{(N-1)}, \ldots, \beta^{(3)}, z) $$
given by the formulas \ref{eq_6.9} and \ref{eq_6.10} in the proposition \ref{prop_6.4}
\end{corollary}

\section{Crystals. Commutor.}
\label{sec_7}
Our goal is to define an action of the Cactus group on the set of regular cell tables defined in \ref{def_6.2}. Using natural bijections constructed in corollaries \ref{cor_6.4} and \ref{cor_6.5}, we obtain the action of the Cactus group on the set of semi-standard short Young tables and the set of Gelfand-Tsetlin patterns for nested orthogonal Lie algebras. 

We will start with the definition of a crystal. One can think of a crystal as a combinatorial model for a representation of the Lie algebra $\mathfrak{g}$. Let $\mathfrak{g}$ be a semisimple Lie algebra. Denote by $P$ its weight lattice and by $P_+$ the set of dominant weights, $I$ the set of vertices of its Dynkin diagram, $\{\alpha_i\}_{i \in I}$  its simple roots and $\{\alpha_i^{\vee}\}_{i \in I}$ its simple co-roots. 

\begin{definition}
\label{def_7.1}
    A \textbf{$\mathfrak{g}$-crystal} is finite set $\mathcal{B}$ along with maps 
    \begin{gather*}
        \text{wt}: \mathcal{B} \rightarrow P \\
        \epsilon_{i}, \phi_{i}: \mathcal{B} \rightarrow \mathbb{Z} \\
        e_i, f_i : \mathcal{B} \rightarrow \mathcal{B} \cup \{0\}
    \end{gather*}
for each $i \in I$ such that: 
\begin{itemize}
    \item for all $b \in \mathcal{B}$ we have $\phi_i(b) - \epsilon_i(b) = \langle\text{wt}(b), \alpha_i^{\vee}  \rangle$
    \item $\epsilon_i(b) = \text{max} \{n: e_i^{n} \cdot b \neq 0 \}$ and $\phi_i(b) = \text{max} \{n : f_i^{n}  \cdot b \neq 0 \}$ for all $b \in \mathcal{B}$ and $i \in I$
    \item if $b \in \mathcal{B}$ and $e_i \cdot b \neq 0$ then $\text{wt} (e_i \cdot b) = \text{wt}(b) + \alpha_i,$ similarly if $f_i \cdot b \neq 0$ then $\text{wt}(f_i \cdot b) = \text{wt}(b) - \alpha_i$ 
    \item for all $b, b^{'} \in \mathcal{B}$ we have $b^{'} = e_i \cdot b$ iff $b = f_i \cdot b^{'}$ 
\end{itemize}
\end{definition}

Looking at this definition it is natural to think of $\mathcal{B}$ as a basis for some representation of $\mathfrak{g}$, with the $e_i$ and $f_i$ representing the actions of Chevalley generators of $\mathfrak{g}$. 

Next, we define the tensor product of the crystals. 

\begin{definition}
\label{def_7.2}
Let $\mathcal{A}, \mathcal{B}$ be crystals. Then \textbf{tensor product} $\mathcal{A} \otimes \mathcal{B}$ defined as follows. The underlying set is $\mathcal{A} \times \mathcal{B}$  and $\text{wt}(a, b) = \text{wt}(a) + \text{wt}(b)$.  We define $e_i$, $f_i$ by the following formula: 

\centering
$e_i \cdot (a, b) = \begin{cases}
(e_i \cdot a, b), \hspace{1mm} \text{if} \hspace{1mm} \epsilon_i(a) > \phi(b)
\\
(a, e_i \cdot b), \hspace{1mm} \text{otherwise}
\end{cases}
$

$f_i \cdot (a, b) = \begin{cases}
(f_i \cdot a, b), \hspace{1mm} \text{if} \hspace{1mm} \epsilon_i(a) \geq \phi(b)
\\
(a, f_i \cdot b), \hspace{1mm} \text{otherwise}
\end{cases}
$
\end{definition}

\begin{definition}
  \label{def_7.3}
We define \textbf{Direct sum} of two crystals as their disjoint union. 
\end{definition}
\begin{definition}
\label{def_7.4}
A morphism of crystals is a map between underlying sets, that commutes with all the structure maps. We assume that it takes $0$ to $0$ even though $0$ is not an element of a crystal.  
\end{definition}

\begin{definition}
\label{def_7.5}
A crystal $\mathcal{B}$ is called \textbf{connected} if the underlying graph (we say that $b$, $b^{'} \in \mathcal{B}$ are joined by the edge if $e_i \cdot b = b^{'}$ for some $i$) is connected. 

A crystal $\mathcal{B}$ is called a \textbf{highest weight crystal of highest weight} $\lambda \in P_+$, if there exists an element $b_\lambda$ (called a \textbf{highest weight element}) such that $\text{wt}(b_\lambda) = \lambda$, $e_i \cdot b_\lambda = 0$ for all $i \in I$  and $\mathcal{B}$ is generated by $f_i$ acting on $b_\lambda$. 
\end{definition}

Clearly, every highest weight crystal is connected. The converse is not true. Moreover, there are non-isomorphic highest weight crystals of the same highest weight. 

However, we can consider only those crystals that naturally arise from representations of $\mathfrak{g}$. For these crystals, it is true that every connected crystal is a highest weight crystal and that highest weight crystals of the same highest weight are isomorphic to each other. 

Formally, let $\mathbb{B} = \{\mathcal{B}_{\lambda} : \lambda \in P_+\}$ be a family of crystals such that $\mathcal{B}_\lambda$ is a crystal of the highest weight $\lambda$ . We call the family $(\mathbb{B}, \iota)$ \textbf{closed}
if $\iota_{\lambda, \mu}: \mathcal{B}_{\lambda + \mu} \rightarrow \mathcal{B}_{\lambda} \otimes \mathcal{B}_{\mu}$ is an inclusion of crystals. 

\begin{theorem} (Due to Joseph, \cite{Joseph}) \label{th_7.1}
There exists a unique closed family of Crystals $(\mathbb{B}, \iota)$. 
\end{theorem}
One of the ways to construct crystals $\mathcal{B}_{\lambda}$ from this closed family is by using crystal bases of $U_q(\mathfrak{g})$-modules . See \cite{Kashiwara} for the details. 

\begin{definition}
\label{def_7.6}
The category of $\mathfrak{g}$-$\text{Crystals}$ is the category whose objects are crystals $\mathcal{B}$ such that every connected component of $\mathcal{B}$ is isomorphic to some $\mathcal{B}_\lambda$ from the closed family. Morphisms in this category are the usual morphisms of crystals. 
\end{definition}

For the category of $\mathfrak{g}$-$\text{Crystals}$ we have the following version of Schur's Lemma. 

\begin{proposition}
\label{prop_7.1}
$Hom(\mathcal{B}_\lambda, \mathcal{B}_\mu)$ contains just the identity map in case $\lambda = \mu$ and is empty otherwise. 
\end{proposition}
\begin{proof}
Since a morphism of crystals commutes with all the structure maps, it sends the highest weight element of $\mathcal{B}_\lambda$ to the highest weight element of $\mathcal{B}_\mu$. Therefore, $\lambda = \mu$. Since any element of $\mathcal{B}_\lambda$ can be obtained by acting on the highest weight element by multiple $f_i$ 's, the morphism is uniquely determined and is equal to the identity map. 
\end{proof}

Category of $\mathfrak{g}$-$\text{Crystals}$ is closed under the tensor product operation. That follows from the fact that tensor product of crystal bases of the $U_q(\mathfrak{g})$-modules is a crystal base of their tensor product and that crystal base for any $U_q(\mathfrak{g})$-module exists and the corresponding crystal $\mathcal{B}$ is uniquely determined by the module (doesn't depend on the choice of the crystal base inside the module). 

Tensor product of crystals is associative: i.e if $\mathcal{B}_1, \mathcal{B}_2, \mathcal{B}_3$ are crystals, then 
$$\alpha_{\mathcal{B}_1, \mathcal{B}_2, \mathcal{B}_3}: (\mathcal{B}_1 \otimes \mathcal{B}_2) \otimes \mathcal{B}_3 \rightarrow \mathcal{B}_1 \otimes (\mathcal{B}_2 \otimes \mathcal{B}_3) $$$$((a, b), c) \mapsto (a, (b, c))$$
is an isomorphism of crystals. 

Note that the tensor product of crystals is not symmetric, i.e. the map 
$$\text{flip} : \mathcal{B}_1 \otimes \mathcal{B}_2 \rightarrow \mathcal{B}_2 \otimes \mathcal{B}_1$$ defined by $(b_1, b_2) \mapsto (b_2, b_1)$ is not a morphism of crystals. 

In the work of Herniques and Kamnitzer \cite{Herniques_Kamnitzer} isomorphisms $$\sigma_{\mathcal{B}_1, \mathcal{B}_2} : \mathcal{B}_1 \otimes \mathcal{B}_2 \rightarrow \mathcal{B}_2 \otimes \mathcal{B}_1$$ were defined for any two $\mathcal{B}_1, \mathcal{B}_2 \in Ob(\mathfrak{g}$-$\text{Crystals})$. This family of maps is called \textbf{commutor}. 

\begin{proposition}
    \label{prop_7.2}
    $\sigma_{\mathcal{B}_1, \mathcal{B}_2} \circ \sigma_{\mathcal{B}_2, \mathcal{B}_1} = 1$.
\end{proposition}

\begin{theorem}
\label{th_7.2} 
The following diagram commutes in $\mathfrak{g}$-$\text{Crystals}$: \\
    
\centering
\begin{tikzcd}[column sep=3em]
\mathcal{B}_1 \otimes \mathcal{B}_2 \otimes \mathcal{B}_3 \arrow[r, "1_{\mathcal{B}_1} \otimes \sigma_{\mathcal{B}_2, \mathcal{B}_3} "] \arrow[d,swap,"\sigma_{\mathcal{B}_1, \mathcal{B}_2} \otimes 1_{\mathcal{B}_3} " ]  &
\mathcal{B}_1 \otimes \mathcal{B}_3 \otimes \mathcal{B}_2 \arrow[d,"\sigma_{\mathcal{B}_1 \otimes \mathcal{B}_3 , \mathcal{B}_2}"] \\
\mathcal{B}_2 \otimes \mathcal{B}_1 \otimes \mathcal{B}_3 \arrow[r,"\sigma_{\mathcal{B}_2 \otimes \mathcal{B}_1 , \mathcal{B}_3}" ]  &  \mathcal{B}_2 \otimes \mathcal{B}_1 \otimes \mathcal{B}_3 
\end{tikzcd}
\end{theorem}
The proofs of \ref{prop_7.2} and \ref{th_7.2} can be found in \cite{Herniques_Kamnitzer}.

\section{Cactus group action}
\label{sec_8}
Let $\mathcal{B}_1, \mathcal{B}_2, \ldots, \mathcal{B}_N \in Ob(\mathfrak{g}$-$\text{Crystals})$. If $1 \leq p \leq r < q \leq N$, we get isomorphisms denoted by $\sigma_{p,q,r}$ defined by the following formula: 
\begin{gather*}
    (\sigma_{p,q,r})_{\mathcal{B}_1 \otimes \mathcal{B_2} \otimes \ldots \otimes \mathcal{B}_N} := 1_{\mathcal{B}_1 \otimes \ldots \otimes \mathcal{B}_{p-1}} \otimes \sigma_{\mathcal{B}_{p} \otimes \ldots \otimes \mathcal{B}_r, \mathcal{B}_{r+1} \otimes \ldots \otimes \mathcal{B}_q} \otimes 1_{\mathcal{B}_{q+1} \otimes \ldots \otimes \mathcal{B}_{N}}: \\
    \mathcal{B}_1 \otimes \ldots \otimes  \mathcal{B}_{p-1} \otimes \mathcal{B}_{p} \otimes \ldots \otimes \mathcal{B}_{r} \otimes \mathcal{B}_{r+1} \otimes \ldots \otimes \mathcal{B}_{q} \otimes \mathcal{B}_{q+1} \otimes \ldots \otimes   \mathcal{B}_N \rightarrow \\ \rightarrow \mathcal{B}_1 \otimes \ldots \otimes  \mathcal{B}_{p-1} \otimes \mathcal{B}_{r+1} \otimes \ldots \otimes \mathcal{B}_{q} \otimes \mathcal{B}_{p} \otimes \ldots \otimes \mathcal{B}_{r} \otimes \mathcal{B}_{q+1} \otimes \ldots \otimes   \mathcal{B}_N
\end{gather*}
We will construct natural isomorphisms using these $\sigma_{p,q,r}$. For $1 \leq p \leq q \leq N$ we define isomorphisms 
\begin{gather*}s_{p, q}: \mathcal{B}_1 \otimes \ldots \otimes  \mathcal{B}_{p-1} \otimes \mathcal{B}_{p} \otimes \mathcal{B}_{p+1} \otimes \ldots \otimes \mathcal{B}_{q-1} \otimes \mathcal{B}_{q}  \otimes \mathcal{B}_{q+1}\otimes \ldots \mathcal{B}_N \rightarrow \\ \rightarrow \mathcal{B}_1 \otimes \ldots \otimes  \mathcal{B}_{p-1} \otimes \mathcal{B}_{q} \otimes \mathcal{B}_{q-1} \otimes \ldots \otimes \mathcal{B}_{p+1} \otimes \mathcal{B}_{p}  \otimes \mathcal{B}_{q+1}\otimes \ldots \mathcal{B}_N 
\end{gather*}
recursively by setting $s_{p, p+1} = \sigma_{p, p, p+1}$ and $s_{p,q} = \sigma_{p, p, q} \circ s_{p+1, q}$ for $q-p > 1$. By convention $s_{p,p} = 1$.

We will also adopt the following notation 
\begin{notation}
  For $p < q$  we denote by $\tilde{s}_{p, q}$ the involutive element of the symmetric group $S_n$ which transposes the segment $[p, q]$. 
  $$\tilde{s}_{p, q} = \begin{pmatrix}
      1 & \ldots & p-1 & p & \ldots & q & q+1 & \ldots & n \\
      1 & \ldots & p-1 & q & \ldots & p & q+1 & \ldots & n 
  \end{pmatrix}$$
\end{notation}

It was proved in \cite{Herniques_Kamnitzer} that these isomorphisms satisfy certain equations. 

\begin{proposition}
Isomorphisms $s_{p,q}$ satisfy the following conditions:
\label{prop_8.1}
    \begin{itemize}
        \item $s_{p, q} \circ s_{p, q} = 1$
        \item $s_{p, q} \circ s_{k, l} = s_{k, l} \circ s_{p, q}$  if  $p < q  < k < l$ 
        \item $s_{p,q} \circ s_{k,l} = s_{m,n} \circ s_{p,q} $ if $p < q$ contains $k < l$, where $m = \tilde{s}_{p,q}(l)$ and $n = \tilde{s}_{p,q}(k)$.  
    \end{itemize}
\end{proposition}

It makes sense to give the definition of the cactus group right after this proposition. 

\begin{definition}
    Let $C_N$ be the group with generators $s_{p,q}$ \hspace{1mm} for $1 \leq p < q \leq N$ and relations: 
    \begin{itemize}
        \item $s_{p,q}^{2} = 1.$
        \item $s_{p,q}s_{k,l} = s_{k,l}s_{p,q}$ if $p < q $ and $k < l$ are disjoint. 
        \item $s_{p,q}s_{k,l} = s_{m,n}s_{p,q}$ if $p < q$ contains $k < l$, where $m = \tilde{s}_{p,q}(l)$ and $n = \tilde{s}_{p,q}(k)$.
    \end{itemize}
\end{definition}
So we see that there is a natural action of a cactus group $C_N$ on the tensor degree of a crystal $\mathcal{B}^{\otimes N}$. The generators of the group $C_N$ act by inner automorphisms of the crystal $\mathcal{B}^{\otimes N}$. These automorphisms essentially permute the connected components of the same highest weight (each connected component is the highest weight subcrystal). 
We need the following proposition.

\begin{proposition}
\label{f_8.1}
If $L_\lambda$ and $L_\beta$ are finite-dimensional irreducible representations of $\mathfrak{g}$ and $\mathcal{B}_\lambda$, $\mathcal{B}_{\beta}$ are the corresponding crystals. Then if 
$$L_\lambda \otimes L_\beta = \bigoplus_{\mu} L_\mu $$ then $$\mathcal{B}_\lambda \otimes \mathcal{B}_\beta \cong \bigoplus_{\mu} \mathcal{B}_{\mu}$$\end{proposition}
\begin{proof}
It is known from \cite{Kashiwara_1} that crystal bases of the same $U_q(\mathfrak{g})$-module $M^{q}$ give rise to isomorphic crystals. Let $L_\lambda^{q}$ be the simple $U_q(\mathfrak{g})$-module of the highest weight $\lambda \in P_+$. If $(\mathcal{L}_\lambda, \mathcal{B}_\lambda)$ is a crystal base of $L_\lambda^{q}$  and $(\mathcal{L}_\beta, \mathcal{B}_\beta)$ is a crystal base of $L_\beta^{q}$, then $(\mathcal{L}_\lambda \otimes \mathcal{L}_\beta, \mathcal{B}_\lambda \otimes \mathcal{B}_\beta)$ is a crystal base of $L_\lambda^{q} \otimes L_\beta^{q}$ (see \cite{Kashiwara_1} for more details). On the other hand, we can decompose $L_\lambda^{q} \otimes L_\beta^{q}$ into the sum of simple $U_q(\mathfrak{g})$-modules: 
$$L_\lambda^{q} \otimes L_\beta^{q} \cong \bigoplus_{\mu} L_\mu^{q}$$
Since direct sum of crystal bases of $U_q(\mathfrak{g})$-modules is a crystal base of the direct sum of these modules, we obtain that $$\mathcal{B}_\lambda \otimes \mathcal{B}_\beta \cong \bigoplus_\mu \mathcal{B}_\mu,$$ where $\mathcal{B}_\mu$ is a crystal arising from the crystal base of $L_\mu^{q}$.  The decomposition into the direct sum of the tensor product of simple $U_q(\mathfrak{g})$-modules agrees with the decomposition of the tensor product of simple $U(\mathfrak{g})$-modules, corresponding to the same highest weights. So, we are done. 

\end{proof}

Then we get the following corollary.
 \begin{corollary}
\label{cor_8.1}
Consider the N-th tensor power of the crystal $\mathcal{B}_{S}$ which is the crystal corresponding to the spinor representation $S$ of the algebra Lie $\mathfrak{g}= \mathfrak{o}_{2n}$. We get the following decomposition of $\mathcal{B}_{S}^{\otimes N}$ into the sum of simple crystals: 

\begin{gather}
    \mathcal{B}_{S}^{\otimes N} \cong \bigoplus_{(\mu_1, \mu_2, \ldots, \mu_N) \in T^N} \mathcal{B}_{\mu_1 + \mu_2 + \cdots + \mu_N} \hspace{1mm},
\end{gather}
where $\hspace{1mm} $$T^N = \{(\mu_1, \mu_2, \ldots, \mu_N) \in P[S]^N \hspace{1mm}| \hspace{1mm} \mu_1 = \omega_{\pm}, \sum_{i=1}^{k} \mu_i \in P_+ \hspace{1mm}, \forall k \leqslant N   \}$
 
 It is evident that connected components of $\mathcal{B}_{S}^{\otimes N}$ isomorphic to the highest weight crystal $\mathcal{B}_{\lambda}$,$\lambda \in \Delta^N $ are indexed by the set $T_\lambda^N$, which has a nice combinatorial interpretation discussed in \ref{prop_6.1}, as the set $\mathfrak{Ctab}(D_\lambda^N)$ . Therefore , there is a natural action of the group $C_N$ on the set $\mathfrak{Ctab}(D_\lambda^N)$. 
 \end{corollary}
 \begin{corollary}
\label{cor_8.2}
Bijections established in \ref{cor_6.4} and \ref{cor_6.5} induce a natural action of the cactus group $C_N$ on the sets $SSSYT(\nu)$ and $\mathfrak{GTP}(\nu)$ for each $\nu \in SYD(N, n)$, defined by the action of the group $C_N$ on the set $\mathfrak{Ctab}(D_\lambda^N)$ for $\lambda = \mathcal{K}_N^{-1}(\mathcal{F}^{-1}(\nu))$, discussed in \ref{cor_8.1}.
  \end{corollary}
Understanding this action will be the subject of a forthcoming paper. 

\renewcommand\refname{References}

\vspace{10mm}

\textsc{Igor \, K. \, Svyatnyy: \, National\, Research\, University\, Higher\, School\, of\, \,Economics, \,Usacheva\, Street 6,\, 119048,\, Moscow,\, Russia.} 

\textit{Email address:} \, \texttt{igor.svyatnyy@outlook.com}

\end{document}